\documentclass[reqno]{amsart}

\usepackage[utf8]{inputenc}
\usepackage{amsfonts}
\usepackage{amssymb}

\usepackage{amsthm}
\usepackage{mathtools}
\usepackage{color}
\usepackage{verbatim}
\usepackage{cases}

\usepackage[mathscr]{eucal}
\usepackage{xspace}

\newcommand{\R}{\mathbb{R}}

\newcommand{\xx}{\mathbf{x}}
\newcommand{\yy}{\mathbf{y}}

\newcommand{\po}{\partial}

\newcommand{\ve}{\varepsilon}

\newcommand{\X}{\times}

\renewcommand{\d}{\delta}

\newcommand{\Div}{{\rm div}\,}

\newcommand{\x}{\mathbf{x}}
\newcommand{\y}{\mathbf{y}}

\newcommand{\m}{\mathbf{m}}
\newcommand{\uu}{\mathbf{u}}

\newcommand{\cO}{{\mathcal O}}

\newcommand{\Sbb}{\mathbb S}

\newcommand{\rhoeta}{[\rho]^\eta}
\newcommand{\rhoveta}{[\rho^{(\ve)}]^\eta}
\newcommand{\rhoneta}{[\rho_n]^\eta}
\newcommand{\rhodeta}{[\rho^{(\delta)}]^\eta}

\theoremstyle{plain}
\newtheorem{theorem}{Theorem}[section]
\newtheorem{corollary}{Corollary}[section]

\newtheorem{lemma}{Lemma}[section]
\newtheorem{proposition}{Proposition}[section]
\theoremstyle{definition}

\theoremstyle{remark}

\newtheorem{remark}{Remark}[section]
\numberwithin{equation}{section}

\author[H. Frid]{Hermano Frid}
 \address{Instituto de Matem\'atica Pura e Aplicada - IMPA\\ Estrada Dona Castorina, 110\\
Rio de Janeiro, RJ, 22460-320, Brazil}
\email{hermano@impa.br}
\thanks{H.~Frid gratefully acknowledges the support from CNPq, through grant proc.~303950/2009-9, and FAPERJ, through grant E-26/103.019/2011.}

\author[D.~Marroquin]{Daniel Marroquin}
\thanks{D.~Marroquin thankfully acknowledges the support from CNPq, through grant proc. 150118/2018-0.}

\address{Instituto de Matem\'{a}tica - Universidade Federal do Rio de Janeiro\\
Av. Athos da Silveira Ramos, 149, Cidade Universit\'{a}ria, Rio de Janeiro, RJ, 21945-970, Brazil}
\email{marroquin@im.ufrj.br}

\author[J.F.C.~Nariyoshi]{Jo\~ao F.C.~Nariyoshi}
 \address{Instituto de Matem\'atica Pura e Aplicada - IMPA\\ Estrada Dona Castorina, 110\\
Rio de Janeiro, RJ, 22460-320, Brazil}
\thanks{J. F. C.~Nariyoshi appreciatively acknowledges the support from CNPq, through grant proc. 140600/2017-5.}
\email{jfcn@impa.br}

\title[Effective viscous flux and compressible Navier-Stokes]
{On Hodge decomposition, effective viscous flux and compressible Navier-Stokes}



\begin{document}

\begin{abstract}
It has been known, since the pioneering works by Serre, Hoff, Va\u \i gant-Kazhikhov, Lions and Feireisl, among others, the regularizing properties of the effective viscous flux and its characterization as the function whose gradient is the gradient part in the Hodge decomposition of the Newtonian force of the  fluid, when the shear viscosity of the fluid is constant. In this article, we 
explore further the connection between the Hodge decomposition of the Newtonian force and the regularizing properties of its gradient part, by addressing the problem of the global existence of weak solutions for compressible Navier-Stokes equations with both viscosities depending on a spatial mollification of the density. 
\end{abstract}

\maketitle

\section{introduction}

The dynamics of a viscous compressible fluid are modeled by the well known Navier-Stokes equations. For a barotropic fluid, the Navier-Stokes equations may be written as
\begin{align}
&\partial_t \rho + \Div(\rho \uu) = 0, \label{1.eq.continuity} \\
&\partial_t(\rho \uu) + \Div (\rho \uu \otimes \uu) + \nabla P = \Div \Sbb,\label{1.eq.momentum}
\end{align}
where $\rho$ and $\uu$ are the density and the velocity field of the fluid, respectively, $P=P(\rho)$ is the pressure, $\Sbb$ is the viscous stress tensor. Note that we are neglecting possible external forces for simplicity.

Let $T>0$ be fixed. Throughout this paper we consider equations \eqref{1.eq.continuity}-\eqref{1.eq.momentum} posed on a smooth bounded open set $\Omega\subseteq \mathbb{R}^N$, with $N\ge 2$, along with the following initial and bounndary conditions:
\begin{align}
&\uu(t,x)=0, & x\in\po\Omega, 0<t<T,\label{1.eq.boundarydata}\\
&\rho(0,x)=\rho_0(x), \rho(0,x)\uu(0,x)=m_0(x), &x\in \Omega.\label{1.eq.initialconds}
\end{align}

Moreover, we assume that the fluid is isentropic and satisfies
\begin{equation}\label{1.eq.pressure}
P(\rho)=A\rho^\gamma,
\end{equation}
for some constants $A>0$ and $\gamma > N/2$. Let us consider the case of a Newtonian fluid, where the viscous stress tensor takes the form
\begin{equation}\label{1.eq.VST}
\Sbb = \lambda\, (\Div \uu) \mathbb{I} + 2\mu\, \mathbb{D}(\uu),
\end{equation}
where $\mathbb{I}$ is the identity matrix in $\mathbb{R}^N$, $\mathbb{D}(\uu)=\frac{1}{2}(\nabla\uu+(\nabla\uu)^\top)$ is the symmetric part of the velocity gradient and $\lambda$ and $\mu$ are the viscosity coefficients which, in general, depend on the density.

We also assume that, in the viscous stress tensor, given by \eqref{1.eq.VST},  the viscosity coefficients are functions of a mollification of the density. More precisely, we denote $\rhoeta=\eta\star \tilde{\rho}$, where $\tilde{\rho}$ is the extension of $\rho$ by $0$ to the whole $\mathbb{R}^N$ and $\eta$ is a smooth function with compact support, and we assume that
\begin{equation}\label{1.eq.viscosities1}
\lambda=\lambda(\rhoeta),\qquad \mu=\mu(\rhoeta),
\end{equation}
where $\zeta\mapsto\lambda(\zeta)$ and $\zeta\mapsto\mu(\zeta)$ are smooth functions satisfying 
\begin{equation}\label{1.eq.viscosities2}
\mu(\zeta)\geq \mu_0 >0,\text{ and } \lambda(\zeta)+\frac{2}{N}\mu(\zeta)\ge 0,\qquad \text{for all }\zeta \in \mathbb{R}.
\end{equation}
Note that if $\eta$ is nonnegative, then we only need to assume \eqref{1.eq.viscosities2} for $\zeta\geq 0$. 

Under these hypotheses, the main purpose of this paper is to prove the following result. 
\begin{theorem} \label{T:1.1}
  Let the initial data satisfy
\begin{equation} \label{1.eq.initialdata}
  \begin{cases}
    \rho_0 \geq 0  \text{ and } \rho_0 \in L^\gamma(\Omega), \\
    \m_0(\x) = 0 \text{ whenever } \rho_0(\x) = 0, \text{ and } \frac{|\m_0|^2}{\rho_0} \in L^1(\Omega). 
  \end{cases}
\end{equation} 
  Then, there exists a weak energy solution $(\rho,\rho\uu)$ to the initial/boundary value problem for the Navier--Stokes equations \eqref{1.eq.continuity}--\eqref{1.eq.momentum} under the conditions expressed in \eqref{1.eq.boundarydata}--\eqref{1.eq.viscosities2}. Furthermore, $(\rho, \rho\uu)$ is a renormalized solution to the continuity equation \eqref{1.eq.continuity}. 
  \end{theorem}
  
  \begin{remark}
  A pair $(\rho,\rho\uu)$ is called a \textit{weak energy solution} to \eqref{1.eq.continuity}-\eqref{1.eq.momentum}, \eqref{1.eq.boundarydata}-\eqref{1.eq.initialconds} if
  \begin{equation*}
    \begin{cases}
      \rho \text{ is nonnegative and} \in L^\infty(0,T;L^\gamma(\Omega)) , \\
      \uu \in L^2(0,T;H_0^1(\Omega;\R^N)),
    \end{cases}
  \end{equation*}
  satisfy \eqref{1.eq.continuity}-\eqref{1.eq.momentum} in the sense of distributions and the following energy inequality holds: for almost every $0<t<T$,
  \begin{multline}\label{1.energy}
  \int_\Omega \Big[ \frac{1}{2} \rho(t,\xx) |\uu (t,\xx)|^2 + \frac{A}{\gamma-1} \rho(t,\xx)^\gamma \Big] \, d\xx \\
  + \int_0^t\int_\Omega \Big[ \mu | \nabla \uu(t',\xx)|^2 + (\lambda + \mu) (\Div \uu(t',\xx))^2  \Big] \, d\xx\, dt'  \\
  \quad\quad\leq \int_\Omega \Big[ \frac{1}{2} |\m_0(\xx)|^2 + \frac{A}{\gamma-1} \rho_0(\xx)^\gamma \Big] \, d\xx \stackrel{\text{def}}{=} E_0. 
  \end{multline}
  \end{remark}
  
  \begin{remark}
  The definition of \textit{renormalized solution} is as follows: $(\rho,\uu)$ is said to be a renormalized solution to the continuity equation \eqref{1.eq.continuity}, if the equation
  \begin{equation*}   
   B(\rho)_t + \Div (B(\rho)\uu) + b(\rho)\Div \uu = 0
  \end{equation*}
 holds in the sense of distributions in $(0,T)\X\R^N$, provided that $\rho$ and $\uu$ are extended to be zero outside of $\Omega$, for any functions
\begin{equation}\label{1.eq.renorm1}
B \in C[0,\infty)\cap C^1(0,\infty),\quad b\in C[0,\infty)\text{ bounded in } [0,\infty),\quad B(0)=b(0)=0,
\end{equation}
satisfying 
\begin{equation}\label{1.eq.renorm2}
b(z)=B'(z)z-B(z).
\end{equation}
\end{remark}

\bigskip

In many cases, it is often convenient to further decompose (the divergence of) the viscous stress tensor in terms of the Hodge decomposition to obtain
\begin{equation}\label{1.eq.momentum1}
\partial_t(\rho \uu) + \Div (\rho \uu \otimes \uu) + \nabla q =  \mathbf{w},
\end{equation}
where
\begin{equation}\label{1.eq.EVF0}
q:=P-\Div\Delta^{-1}\Div\Sbb,
\end{equation}
and $\mathbf{w} :=\Div \Sbb - \nabla(\Div\Delta^{-1}\Div\Sbb)$, which satisfies $\Div \mathbf{w} =0$. Here, the operator $\Div\Delta^{-1}\Div$ may be regarded through its components in terms of the Riesz transform as
\[
\Div\Delta^{-1}\Div\Sbb = (\nabla\Delta^{-1}\nabla):\mathbb{S} = \sum_{i,j=1}^N R_iR_j\Sbb_{ij},
\]
where, $R_i$ is the operator whose Fourier symbol is $-i\frac{\xi_i}{|\xi|}$ (see Appendix \ref{ApdxC}). The function $q$ is often referred to in \cite{F} as the {\sl effective viscous pressure} (see, e.g., \cite{F}). 

 Then, we note that $\Div\Delta^{-1}\Div\mathbb{D}(\uu)= \Div\uu$ and so
\begin{multline}\label{1.eq.EVF}
q=P-(\lambda + 2\mu )\Div \uu + 2 [\mu, \nabla\Delta^{-1}\nabla]:\mathbb{D}(\uu)\\
=P-(\lambda + 2\mu )\Div \uu + 2 \sum_{i,j=1}^N [\mu;R_iR_j]\mathbb{D}_{ij}( \uu),
\end{multline}
where $[b;R_iR_j ](f_{ij}) = b R_iR_j(f_{ij}) - R_iR_j(bf_{ij})$ is the
Lie bracket.

Note that if the shear viscosity $\mu$ is constant, then 
$$
q=q_0 := P - (\lambda+2\mu)\Div \uu, 
$$
which is  usually called the {\sl effective viscous flux}. 

It has been known, since the pioneering works by Serre \cite{S}, Hoff \cite{H1,H2}, Va\u \i gant-Kazhikhov \cite{VK}, Lions \cite{L2} and Feireisl \cite{F}, among others, the regularizing properties of the effective viscous flux $q_0$, when the shear viscosity is constant.

Reasoning as in \cite{F2}, applying the operator $\Div\Delta^{-1}$ to equation \eqref{1.eq.momentum1} we see that
\[
q=-\Div\Delta^{-1}\partial_t(\rho\uu) - \Div\Delta^{-1}\Div(\rho\uu\otimes\uu),
\]
and thus, formally, multiplying by $\rho$ and using the continuity equation, it holds
\begin{multline}\label{1.eq.EVFrho}
q\rho = -\partial_t\big(\Div\Delta^{-1}[\rho\uu]\rho \big) - \Div\big( \rho\uu\, \Div\Delta^{-1}[\rho\uu] \big) \\
+\rho\uu\, \Div\Delta^{-1}\Div [\rho\uu]- \Div\Delta^{-1}\Div[\rho\uu\otimes\uu]\,\rho.
\end{multline}
Now, the right-hand side of this identity can be shown to be weakly continuous in the following sense. Let $(\rho^\ve,\rho^\ve\uu^\ve)$ be a sequence of solutions of the Navier-Stokes equations \eqref{1.eq.continuity}-\eqref{1.eq.momentum} (or a convenient approximation of them) such that $\rho^\ve\to \rho$, $\rho^\ve\uu^\ve\to \rho\uu$ and $\nabla\uu^\ve\to\nabla\uu$ weakly in a suitable $L^p$ space, related to the natural a priori estimates. 
If the viscosity coefficients are constant, it can been shown that $(\rho,\rho\uu)$ is a weak solution of the Navier-Stokes equations with $P=\overline{P(\rho)}$, where the overline stands for a weak limit of the sequence indexed by $\ve$. Then, considering identity \eqref{1.eq.EVFrho} for both $(\rho^\ve,\rho^\ve\uu^\ve)$ and $(\rho,\rho\uu)$, each of the terms on the right-hand side converges weakly to its counterpart as $\ve\to 0$. Thus, it follows that
\begin{equation}\label{1.eq.EVF-Lions}
\overline{\Big(P(\rho)-(\lambda + 2\mu )\Div \uu\Big)\rho} =\Big(\overline{P(\rho)} +(\lambda+2\mu)\Div\uu\Big)\rho.
\end{equation}
This remarkable identity, first discovered by P.-L.~Lions (see \cite{L2}), is the key point to show existence of global weak solutions for the Navier-Stokes equations. The main issue is to show that the sequence of densities converges strongly in order to account for the nonlinearity of the pressure term, that is, to show that $\overline{P(\rho)}=P(\rho)$. 

Note that, if the viscosity coefficients are functions of $\rho$, then the (formal) reasoning outlined above does not yield identity \eqref{1.eq.EVF-Lions}, as weak convergence, in general, does not commute with products.  Also, the presence of the commutator in identity \eqref{1.eq.EVF} requires careful analysis. On the bright side, the regularity results by Coifman and Meyer \cite{CM} and Coifman, Rochberg and Weiss \cite{CRW} regarding commutators of Riesz transforms and the operators of multiplication shed some light into the possibility of extending these results to more general settings.
It is the purpose of this paper to further explore these ideas  for non-constant viscosities satisfying \eqref{1.eq.viscosities1}-\eqref{1.eq.viscosities2}  to prove Theorem~\ref{T:1.1}.

The problem of existence of global solutions to the compressible Navier-Stokes equations with  viscosity coefficients depending on the density is a difficult problem, specially in dimensions greater than one. The general theory developed by Lions in \cite{L2}, and later extended by Feireisl, Novotn\'{y} and Petzeltov\'{a} in \cite{FNP} and by Feireisl \cite{F}, who extended the range of the exponent of the pressure law to the optimal value, relies on a certain continuity with respect to weak convergence of the effective viscous flux \eqref{1.eq.EVF-Lions}, which, so far, for space dimension $\ge3$,  has only been proved in the case of constant viscosity coefficients. In dimension $2$, Va\u \i gant and Kazhikhov \cite{VK} studied the periodic case under the assumptions that $\mu$ is constant and $\lambda(\rho)=\rho^\beta$, with $\beta>3$. Of course, in this case the term involving the commutator in identity \eqref{1.eq.EVF} equals zero and  so $q=q_0$.
Their result relies heavily on the regularity of the effective viscous flux and its identification with the function whose gradient is the gradient part of the Hodge decomposition of the Newtonian force of the fluid, which corresponds to \eqref{1.eq.EVF0}. One great advantage of this decomposition in dimension $2$ is that the divergence-free part of the Hodge decomposition $w$ may be written as $\nabla^\perp G$, for some function $G$, which allows for the deduction of higher order regularity estimates on the solutions when combined with the periodic boundary conditions. 

A significant breakthrough on the construction of weak solutions with density dependent viscosities has been made by Bresch and Desjardins in a series of papers \cite{BD1,BD2,BD3,BD4,BD5} and by Bresch, Desjardins and Lin \cite{BDL}. In general, when the viscosity coefficients depend on the density, the equations may become degenerate when close to the vacuum and, in particular, the velocity field is not bounded in $L^2(0,T;H^1(\Omega))$. However, this degeneracy provides a very particular structure that yields some integrability properties of the gradient of the density, under some restrictions that relate the viscosity coefficients. Namely the relation
\[
\lambda(\zeta)=2(\zeta\mu'(\zeta)-\mu(\zeta)),
\]	 
known as the Bresch-Desjardins relation, which was introduced in \cite{BD4}. In \cite{VY}, Vasseur and Yu \cite{MV} proved existence of solutions for the compressible barotropic Navier–Stokes equations when $\mu(\rho)=\rho$ and $\lambda=0$ using a stability result by Mellet and Vasseur \cite{MV}. Independently and using different methods, Li and Xin \cite{LX} also established existence of solutions for the compressible barotropic Navier–Stokes equations with density dependent viscosities covering the case considered by Vasseur and Yu, as well as more general viscosities satisfying the
Bresch-Desjardins relation, but with a non-symmetric stress diffusion ($\mathbb{S}=\mu(\rho)\nabla\uu+ \lambda(\rho)\Div\uu$) and several extra restrictions on the viscosities and the pressure law. The case of the heat-conductive fluids has been treated by Bresch and Desjardins \cite{BD3}. More recently, Bresch, Vasseur and Yu \cite{BVY} extended the previous results by Vasseur and Yu and by Li and Xin to more general assumptions on the viscosities, satisfying the Bresch-Desjardins relation, including, in particular, the case when $\mu(\rho)=\mu_0\rho^\alpha$ with $\mu_0>0$ and $2/3<\alpha<4$. All of these results, regarding non-constant viscosity coefficients, have been posed either in $\mathbb{R}^3$ or in the torus $\mathbb{T}^3$ and they all rely on the higher regularity on the density allowed by the degeneracy of the viscosity.  None of them, however, uses the regularity of the effective viscous flux discovered by Lions, as it is only available for the case of constant viscosities so far.

The core result of this paper is  the extension of the weak continuity of the effective viscous flux, first proved by Lions~\cite{L2}, to the case where the viscosities may depend on  local spatial averages of the density as  in \eqref{1.eq.viscosities1}. Then, we  apply it to prove Theorem~\ref{T:1.1}. 
The strategy of the proof of Theorem~\ref{T:1.1} follows the same lines  of Feireisl in \cite{F}. We find solutions of the Navier-Stokes equations as weak limits of a sequence of solutions of a regularized system, where the main difficulty is to show strong convergence of the densities in order to account for the pressure term. At this point, aside from the notion of renormalized solutions, the key tool, and main result of this paper, is an identity which extends \eqref{1.eq.EVF-Lions} to the case we consider here (see Theorem~\ref{2.thm.EVF} below). Namely,
\begin{multline}\label{1.eq.EVF-identity}
\overline{\left(P(\rho)-\big( \lambda(\rhoeta) + 2\mu(\rhoeta)\big)\Div \uu\right)\rho} \\=\left(\overline{P(\rho)} +\big( \lambda(\rhoeta) + 2\mu(\rhoeta)\big)\Div\uu\right)\rho.
\end{multline}
This identity is  crucial to prove that the limit functions $(\rho,\uu)$ solve the continuity equation in the sense of renormalized solutions. This last bit of information may be used in order to prove the strong convergence of the densities. However, at this last stage we also need the following identity regarding the effective viscous flux
\begin{equation}\label{1.eq.EVF-identity2}
\overline{\left(\frac{P(\rho)}{\lambda(\rhoeta) + 2\mu(\rhoeta)}-\Div \uu\right)\rho} =\left(\frac{\overline{P(\rho)}}{\lambda(\rhoeta) + 2\mu(\rhoeta)} +\Div\uu\right)\rho.
\end{equation}
Note that when the viscosity coefficients are constant both \eqref{1.eq.EVF-identity} and \eqref{1.eq.EVF-identity2} reduce to \eqref{1.eq.EVF-Lions}.

The rest of this paper is organized as follows. In Section~\ref{S.EVF} we state precisely and prove the main ingredient for the proof   of Theorem~\ref{T:1.1}, that is,  Theorem~\ref{2.thm.EVF}, which is about the weak continuity of the effective viscous flux. Let us point that the key to the proof of this result is the realization that the effective viscous pressure may be written as \eqref{1.eq.EVF}. 
Then, as an application of Theorem~\ref{2.thm.EVF}, we prove Theorem~\ref{T:1.1} by adapting the framework contained in \cite{F}. More precisely, in Section~\ref{S.regularized} we introduce a two-level regularization of the Navier-Stokes depending on two small parameters $\ve$ and $\delta$, which correspond to an artificial viscosity added to the continuity equation and an artificial pressure term, respectively. We discuss the solvability of the regularized system and state some a priori estimates on the solutions. In Section~\ref{S.veto0} we show convergence of the solutions as the artificial viscosity tends to $0$, by means of a variant of Theorem~\ref{2.thm.EVF}. Finally, in Section~\ref{S.dto0} we show convergence of the solutions as the artificial pressure vanishes using Theorem~\ref{2.thm.EVF} once again, which is the last step of the proof of Theorem~\ref{T:1.1}.

\section{The effective viscous flux}\label{S.EVF}

Let us consider a sequence $(\rho_n,\uu_n)$ of weak energy solutions of \eqref{1.eq.continuity}-\eqref{1.eq.momentum} with $P$ given by \eqref{1.eq.pressure}, $\mathbb{S}$ given by \eqref{1.eq.VST} and whose viscosity coefficients satisfy \eqref{1.eq.viscosities1} and \eqref{1.eq.viscosities2}. Let us also assume that $\rho_n$ and $\uu_n$ satisfy the continuity equation \eqref{1.eq.continuity} in the sense of renormalized solutions and that there are some function $\rho$, $\uu$ and $\overline{P}$ such that
\begin{equation}\label{2.convergen}
\begin{cases}
\rho_n\to \rho &\text{weakly--}\star\text{ in }L^{\infty}(0,T;L^\gamma(\Omega)),\\
\uu_n\to\uu &\text{weakly in }L^2(0,T;H_0^1(\Omega)),\\
P(\rho_n) \to \overline{P} &\text{weakly in } L^r((0,T)\times\Omega), \text{ for some $r>1$},
\end{cases}
\end{equation}
as $n\to\infty$. Moreover, let us assume that 
\begin{equation}\label{2.rhoun2}
\rho_n|\uu_n|^2 \ \text{is bounded in }L^\infty(0,T;L^1(\Omega)).
\end{equation}

Then, we have the following result, which establishes the announced extension of  \eqref{1.eq.EVF-identity}, and represents a decisive step in the proof of  the weak continuity of the effective viscous flux.

\begin{theorem}\label{2.thm.EVF}
Let $\gamma>\frac{N}{2}$. Then, passing to a subsequence if necessary  
\begin{multline}\label{2.eq.thmEFV}
      \int_0^T \int_\Omega \varphi B(\rho_n) \Big[ \frac{P(\rho_n)}{2\mu([\rho_n]^\eta) + \lambda([\rho_n]^\eta)} - \Div \uu_n \Big] \, d\xx dt\\
               \rightarrow \int_0^T \int_\Omega \varphi \overline{B(\rho)} \Big[ \frac{\overline{P}}{2\mu([\rho]^\eta) + \lambda([\rho]^\eta)} - \Div \uu \Big] \, d\xx dt,
    \end{multline}        
for any $\varphi\in C_c^\infty((0,T)\times\Omega)$,  and any bounded and continuous function $B$, where $B(\rho_n)\to \overline{B(\rho)}$ weakly-$\star$ in $L^\infty((0,T)\times\Omega)$. Likewise, we also have that
\begin{multline}\label{2.eq.thmEFV2}
      \int_0^T \int_\Omega \varphi B(\rho_n) \Big[ P(\rho_n) - \Big(2\mu([\rho_n]^\eta) + \lambda([\rho_n]^\eta)\Big)\Div \uu_n \Big] \, d\xx dt\\
               \rightarrow \int_0^T \int_\Omega \varphi \overline{B(\rho)} \Big[ \overline{P} - \Big(2\mu([\rho]^\eta) + \lambda([\rho]^\eta)\Big)\Div \uu \Big] \, d\xx dt,
    \end{multline} 
\end{theorem}

The proof of  \eqref{2.eq.thmEFV} is similar to that of \eqref{2.eq.thmEFV2}, therefore we only prove in details  the former and comment briefly the modifications that have to be made for the proof of the latter.  Before proceeding to the proof properly, we give a brief account of the strategy of the proof as follows. First, we rewrite the momentum equation satisfied by the sequence $(\rho_n,\uu_n)$ as
\begin{equation}\label{2.eq.momentumn}
 \nabla P(\rho_n) - \Div \Sbb_n = -\partial_t(\rho_n \uu_n) - \Div (\rho_n \uu_n \otimes \uu_n) ,
\end{equation}
and show that we may take the limit as $n\to\infty$ to conclude that the following equation holds in the sense of distributions
\begin{equation}\label{2.eq.momentumlimit}
 \nabla \overline{P} - \Div \Sbb = -\partial_t(\rho \uu) - \Div (\rho \uu \otimes \uu),
\end{equation}
where, $\mathbb{S}=\lambda(\rhoeta)(\Div\uu)\mathbb{I}+2\mu(\rhoeta)\mathbb{D}(\uu)$.

Then, we take the test functions
  \begin{equation} \label{2.eqphi}
   \begin{cases}
    \phi_n(t,\xx) = \varphi(t,\xx) \Delta^{-1} \nabla \Big( B(\rho_n) F([\rho_n]^\eta) \Big)(t,\xx) \text{ and} \\  
    \phi(t,\xx) = \varphi(t,\xx) \Delta^{-1} \nabla \Big( \overline{B(\rho)} F([\rho]^\eta) \Big)(t,\xx)
   \end{cases}
  \end{equation}
  in equations \eqref{2.eq.momentumn} and \eqref{2.eq.momentumlimit}, respectively, where
 \begin{equation}
    F(\xi) = \frac{1}{\lambda(\xi) + 2\mu(\xi)}. \label{2.defF}
  \end{equation}
After some manipulation, upon integrating by parts, the terms involved in the conclusion of Theorem~\ref{2.thm.EVF} will appear in each one of the two resulting equations. Then, to conclude it suffices to show that each of the remaining terms converges to its counterpart as $n \to \infty$. 

Now, in order to deduce equation \eqref{2.eq.momentumlimit} we first observe that 
\begin{equation}\label{2.eq.rhontorho}
\rho_n\to \rho\text{ in }C([0,T];L_{weak}^\gamma(\Omega)),
\end{equation} 
where the space $C([0,T];L_{\text{weak}}^p(\Omega))$ is defined as in Appendix~\ref{A.ApdxA}. Indeed, \eqref{2.eq.rhontorho} follows by applying Proposition~\ref{A.B.prop2} in view of the equations verified by each $\rho_n$. Since the convolution operator $f\in L^\gamma(\Omega)\mapsto [f]^\eta \in C^k(\overline{\Omega})$ is compact for any integer $k\ge 0$, we see that 
\[
\lambda([\rho_n]^\eta)\to \lambda(\rhoeta) \text{ and } \mu([\rho_n]\eta)\to \mu(\rhoeta)\quad \text{strongly in }C(0,T;C^\infty(\overline{\Omega})).
\]
Thus, if $\mathbb{S}_n$ denotes the viscous stress tensor corresponding to $(\rho_n,\uu_n)$, we have that $\mathbb{S}_n\to \mathbb{S}$ in the sense of distributions. Moreover, by \eqref{2.convergen} and \eqref{2.rhoun2} we have that the sequence $\rho_n\uu_n$ is bounded in $L^\infty(0,T;L^{\frac{2\gamma}{\gamma+1}}(\Omega))$, so we can apply Proposition~\ref{A.B.prop2} once again to deduce that (up to a subsequence) $\rho_n\uu_n$ converges in $C([0,T];L_{weak}^{\frac{2\gamma}{\gamma+1}}(\Omega))$. To show that the limit equals $\rho\uu$ it suffices to note that $L^\gamma(\Omega)$ is compactly imbedded in $H^{-1}(\Omega)$, so that
\[
\rho_n\to \rho \text{ in }C([0,T];H^{-1}(\Omega)),
\]
which, together with \eqref{2.convergen}, yields the convergence 
\begin{equation}\label{2.rhonuntorhouweak}
\rho_n\uu_n\to \rho\uu\quad\text{in }C([0,T];L_{weak}^{\frac{2\gamma}{\gamma+1}}(\Omega)).
\end{equation}
Now, note that $\frac{2\gamma}{\gamma+1}$ is also compactly imbedded in $H^{-1}(\Omega)$ due to our hypothesis that $\gamma > N/2$ (see e.g. theorem~2.8 in \cite{F}). Thus, we conclude that 
\begin{equation}\label{2.eq.rhonunstrong}
\rho_n\uu_n\to\rho\uu\quad\text{in }C([0,T];H^{-1}(\Omega)),
\end{equation}
and,  consequently $\rho_n\uu_n\otimes\uu_n\to \rho\uu\otimes\uu$ in the sense of distributions. Furthermore, since $H_0^1(\Omega)$ is imbedded continuously in $L^{c}(\Omega)$, where $c=\frac{2N}{N-2}$ if $N\ge 3$ and $c>1$ arbitrary if $N=2$, we have that, in fact, 
\[
\rho_n\uu_n\otimes\uu_n \to \rho\uu\otimes\uu\quad\text{weakly in }L^2(0,T;L^{s}(\Omega)),
\] 
where $\frac{1}{s}=\frac{\gamma+1}{2\gamma}+\frac{N-2}{2N}$ if $N\ge 3$ and $\frac{1}{s}>\frac{\gamma+1}{2\gamma}$ if $N=2$. In this way, equation \eqref{2.eq.momentumlimit} results by taking the limit as $n\to\infty$ in equation \eqref{2.eq.momentumn}.

In order to proceed with the remaining details of the proof of Theorem~\ref{2.thm.EVF}, we first need a couple of preliminary observations. The first one concerns the resulting terms in the left-hand side of equations \eqref{2.eq.momentumn} and \eqref{2.eq.momentumlimit} after taking the test functions $\phi_n$ and $\phi$.

\begin{lemma}\label{2.l.eval}
Let $F$ be given by \eqref{2.defF}. Then,
\begin{multline*}
\int_0^T\int_\Omega \overline{P} \Div\phi\, d\x dt - \int_0^T\int_\Omega \mathbb{S}\cdot \nabla\phi\, d\x dt \\
= \int_0^T\int_\Omega \varphi q\overline{B(\rho)} F([\rho]^\eta)d\x dt - \int_0^T\int_\Omega [\nabla\Delta^{-1}\nabla,\varphi]:(\mathbb{S})\overline{B(\rho)} F([\rho]^\eta)d\x dt\\
   + \int_0^T\int_\Omega \Big((\overline{P}\ \mathbb{I} -\mathbb{S})\cdot\nabla\varphi\Big)\cdot \Delta^{-1}\nabla \Big(\overline{B(\rho)} F([\rho]^\eta)\Big)d\x dt,
\end{multline*}
where, $q$ is the effective viscous pressure, corresponding to $\overline{P}$ minus the gradient part of the Hodge decomposition of $\Div\mathbb{S}$, that is
\begin{equation}\label{2.eq.EVP}
q:=\overline{P}-\Big(\lambda([\rho]^\eta)+2\mu([\rho]^\eta)\Big)\Div\uu + 2[\mu(\rhoeta),\nabla\Delta^{-1}\nabla]:\mathbb{D}(\uu).
\end{equation}

Similarly,
\begin{multline*}
\int_0^T\int_\Omega P(\rho_n) \Div\phi\, d\x dt -\int_0^T\int_\Omega \mathbb{S}_n\cdot \nabla\phi_n \, d\x dt \\
= \int_0^T\int_\Omega \varphi q_n B(\rho_n) F([\rho_n]^\eta)d\x dt - \int_0^T\int_\Omega [\nabla\Delta^{-1}\nabla,\varphi]:(\mathbb{S}_n)B(\rho_n) F([\rho_n]^\eta)d\x dt\\
   + \int_0^T\int_\Omega \Big((P(\rho_n)\ \mathbb{I}-\mathbb{S}_n)\cdot\nabla\varphi\Big)\cdot \Delta^{-1}\nabla \Big(B(\rho_n) F([\rho_n]^\eta)\Big)d\x dt,
\end{multline*}
where, 
\begin{equation}\label{2.eq.EVPn}
q_n:=P(\rho_n)-\Big(\lambda([\rho_n]^\eta)+2\mu([\rho_n]^\eta)\Big)\Div\uu_n + 2[\mu(\rhoneta),\nabla\Delta^{-1}\nabla]:\mathbb{D}(\uu_n).
\end{equation}
\end{lemma}

\begin{proof}
The proof of both identities is similar and therefore we only prove the first one. To that end, note that on the one hand we have that
\begin{multline}\label{3.eq.l1}
\int_0^T\int_\Omega \mathbb{S}\cdot \nabla\phi d\x dt = \int_0^T\int_\Omega \varphi\mathbb{S}: \nabla\Delta^{-1}\nabla \Big(\overline{B(\rho)} F([\rho])\Big)d\x dt\\
  \qquad\qquad\qquad + \int_0^T\int_\Omega \Big(\mathbb{S}\cdot\nabla\varphi\Big)\cdot \Delta^{-1}\nabla \Big(\overline{B(\rho)} F([\rho])\Big)d\x dt
\end{multline}
On the other hand, using the selfadjointness of the operator $\nabla\Delta^{-1}\nabla$ and by the definition of the commutator $[\cdot,\cdot]$ we see that
\begin{multline*}
\int_0^T\int_\Omega \varphi\mathbb{S}: \nabla\Delta^{-1}\nabla \Big(\overline{B(\rho)} F([\rho])\Big)d\x dt = \int_0^T\int_\Omega [\nabla\Delta^{-1}\nabla,\varphi](\mathbb{S})\overline{B(\rho)} F([\rho])d\x dt \\
+ \int_0^T\int_\Omega \varphi \tilde{q}\rho F([\rho]^\eta)d\x dt
\end{multline*}
where, $\tilde{q}:=\nabla\Delta^{-1}\nabla:\Sbb$. Then, noting that $\nabla\Delta^{-1}\nabla:\mathbb{D}(\uu)= \Div\uu$ we see that $\tilde{q}=\Big(\lambda([\rho]^\eta)+2\mu([\rho]^\eta)\Big)\Div\uu + 2[\mu(\rhoeta),\nabla\Delta^{-1}\nabla]:\mathbb{D}(\uu)$.

Finally, we see that
\begin{multline}\label{3.eq.l3}
\int_0^T\int_\Omega \overline{P}\Div\phi d\x dt = \int_0^T\int_\Omega \overline{P}\, \nabla\varphi\cdot \Delta^{-1}\nabla\Big( \overline{B(\rho)} F(\rhoeta) \Big) d\x dt\\
 + \int_0^T\int_\Omega \overline{P}\varphi \overline{B(\rho)} F(\rhoeta) d\x dt,
\end{multline}
and the result follows by gathering \eqref{3.eq.l1}-\eqref{3.eq.l3}.
\end{proof}

The second preliminary observation will allow us to evaluate the first term on the right-hand side of equations \eqref{2.eq.momentumn} and \eqref{2.eq.momentumlimit} after taking the test functions $\phi_n$ and $\phi$. Indeed, In order to evaluate the partial derivative of $\phi_n$ and $\phi$ with respect to $t$ we need to deduce the equations that $B(\rho_n) F(\rhoneta)$ and $\overline{B(\rho)} F(\rhoeta)$ satisfy. 

Let us point out that throughout this section the over line stands for a weak limit of the sequence indexed in $n$.

\begin{lemma}\label{2.l.eqBF}
Let $F$ be given by \eqref{2.defF} and let $b$ be given by \eqref{1.eq.renorm2}.Then, the following equation holds in the sense of distributions in $\mathbb{R}^N$
\begin{equation}\label{2.eq.BnFn}
 \Big(B(\rho_n) F(\rhoneta)\Big)_t + \Div \Big(B(\rho_n) F(\rhoneta )  \uu_n  \Big)  = h_n ,
\end{equation}
where
\begin{multline*}
h_n \stackrel{\text{def}}{=} -F(\rhoneta)b(\rho_n)\Div\uu_n +  B(\rho_n)  F'(\rhoneta)\Big[\Div \big(\rhoneta \uu_n  \big) - \Div \big(\eta \star (\rho_n  \uu_n )\big) \Big]\\
  - B(\rho_n)\rhoneta F'(\rhoneta) \Div \uu_n
\end{multline*}
  
Moreover, $h_n\to h$ weakly in $ L^2((0,T) \times \R^N)$, as $n\to\infty$, where
\begin{multline*}
    h \stackrel{\text{def}}{=}-F(\rhoeta)\overline{b(\rho)\Div\uu} +  F'(\rhoeta)\Big[ \overline{B(\rho)} \nabla\rhoeta  \cdot \uu + \rhoeta\, \overline{B(\rho) \Div \uu}   \\
  - \overline{B(\rho)}\Div \big(\eta \star (\rho \uu )\big) \Big]     - \rhoeta\, F'(\rhoeta)\overline{B(\rho)\Div\uu},
  \end{multline*}
and $(\rho, \uu)$ satisfy the following equation in the sense of distributions
\begin{equation}\label{2.eq.BF}
    \big(\overline{B(\rho)} F(\rhoeta)\big)_t + \Div \big(\overline{B(\rho)} F(\rhoeta )  \uu \big)  =\, h.
 \end{equation}
\end{lemma}

\begin{proof}
 \textit{Step \#1.} By Proposition~\ref{A.B.prop2} we have that $B(\rho_n)\to \overline{B(\rho)}$ in $C([0,T];L_{weak}^\gamma(\Omega))$. Then, since $L^\gamma(\Omega)$ is compactly imbedded in $H^{-1}(\Omega)$ due to our assumption that $\gamma > N/2$, we conclude that
\[
B(\rho_n)\uu_n \to \overline{B(\rho)}\, \uu
\]
in the sense of distributions. Then, we can take the limit as $n\to \infty$ in the equation
\begin{equation} \label{2.eq.Bn}  
   B(\rho_n)_t + \Div (B(\rho_n)\uu_n) + b(\rho_n)\Div \uu_n = 0,
  \end{equation}
with $b$ given by \eqref{1.eq.renorm2}, to obtain that the following equation is satisfied in the sense of distributions
\begin{equation*}
   \overline{B(\rho)}_t + \Div (\overline{B(\rho)}\, \uu) + \overline{b(\rho)\Div \uu }= 0.
  \end{equation*}
Here, $\overline{b(\rho)\Div \uu }$ is a weak limit of the sequence $b(\rho_n)\Div \uu_n$.

\textit{Step \#2.} In order to deduce \eqref{2.eq.BnFn} we proceed as follows. First, extending $\rho_n$ and $\uu_n$ by zero outside $\Omega$, we may take $\eta(\x-\y)$ as a test function in the continuity equation
\begin{equation*}
\partial_t\rho_n + \Div(\rho_n\uu_n) = 0
\end{equation*}
to deduce an equation for $\rhoneta$. Then, we use this equation to deduce an equation for $F(\rhoveta)$. Finally, we use the resulting equation and combine it with equation \eqref{2.eq.Bn} to conclude. We omit the details.

\textit{Step \#3.} Lastly, in order to deduce equation \eqref{2.eq.BF} it suffices to take the limit as $n\to\infty$ in equation \eqref{2.eq.BnFn}, noting that each term converges weakly to its  counterpart.
\end{proof}

\begin{proof}[Proof of Theorem~\ref{2.thm.EVF}]
 \textit{Step \#0.} Let us first mention that, as pointed out in \cite{F}, it is enough to prove the  Theorem for functions $B$ in $C[0,\infty)\cap C^1(0,\infty)$ which are bounded and such that $b(z)$, given by \eqref{1.eq.renorm2}, is also continuous in $[0,\infty)$, bounded and satisfy $B(0)=b(0)=0$; that is, for functions that satisfy the conditions \eqref{1.eq.renorm1} and \eqref{1.eq.renorm2} of the definition of renormalized solutions. Indeed, we can always approximate any bounded and continuous function $B$ by a sequence $B_m$, $m=1,2,...$, of functions that satisfy these requirements, which are bounded uniformly with respect to $m$ and converge to $B-B(0)$ on compact subsets of $\mathbb{R}$, and then pass to the limit as $m\to\infty$ in \eqref{2.eq.thmEFV}.
 
\textit{Step \#1.} Let us prove \eqref{2.eq.thmEFV}. Let $\phi_n$ be given by \eqref{2.eqphi}. Then, seeing that 
\[\partial_t\phi = \varphi(t,\xx) \Delta^{-1} \nabla \partial_t\Big( B(\rho_n) F([\rho_n]^\eta) \Big) + \partial_t\varphi\ \Delta^{-1} \nabla \Big( B(\rho_n) F([\rho_n]^\eta) \Big),
\]
 we take $\phi_n$ as a test function in equation \eqref{2.eq.momentumn} to obtain
\begin{equation}\label{2.idn}
\int_0^T\int_\Omega P(\rho_n) \Div\phi_n\, d\x dt -\int_0^T\int_\Omega \mathbb{S}_n\cdot \nabla\phi_n \, d\x dt = \sum_{j=1}^5 I_j^n,
\end{equation}
where,
\begin{align*}
&I_1^n=\int_0^T \int_\Omega \varphi\rho_n\uu_n \Delta^{-1} \nabla\Div \Big( B(\rho_n) F([\rho_n]^\eta)\uu_n \Big)d\x dt,\\
&I_2^n=-\int_0^T \int_\Omega \varphi\rho_n\uu_n \Delta^{-1} \nabla h_n d\x dt,\\
&I_3^n = -\int_0^T \int_\Omega \rho_n\uu_n \varphi_t\Delta^{-1} \nabla \Big( B(\rho_n) F([\rho_n]^\eta)\Big) d\x dt,\\
&I_4^n= - \int_0^T\int_\Omega (\rho_n\uu_n\otimes\uu_n\cdot \nabla\varphi)\cdot\Delta^{-1} \nabla \Big( B(\rho_n) F([\rho_n]^\eta)\Big) d\x dt,\\
&I_5^n=-\int_0^T\int_\Omega \varphi\rho_n\uu_n\otimes\uu_n : \nabla\Delta^{-1} \nabla \Big( B(\rho_n) F([\rho_n]^\eta)\Big) d\x dt.
\end{align*}

Note that we have used Lemma~\ref{2.l.eqBF} in order to deal with the time derivative of $ B(\rho_n) F([\rho_n]^\eta)$, giving rise to $I_1^n$ and $I_2^n$.

Similarly, taking $\varphi$, given by \eqref{2.eqphi}, as a test function in equation \eqref{2.eq.momentumlimit} and using Lemma~\ref{2.l.eqBF} once again, we obtain that
\begin{equation}\label{2.idlimit}
\int_0^T\int_\Omega \overline{P}\,  \Div\phi\, d\x dt -\int_0^T\int_\Omega \mathbb{S}\cdot \nabla\phi \, d\x dt = \sum_{j=1}^5 I_j,
\end{equation}
where,
\begin{align*}
&I_1=\int_0^T \int_\Omega \varphi\rho\uu \Delta^{-1} \nabla\Div \Big( \overline{B(\rho)} F([\rho]^\eta)\uu \Big)d\x dt,\\
&I_2=-\int_0^T \int_\Omega \varphi\rho\uu \Delta^{-1} \nabla h d\x dt,\\
&I_3 = -\int_0^T \int_\Omega \rho\uu \varphi_t\Delta^{-1} \nabla \Big( \overline{B(\rho)} F([\rho]^\eta)\Big) d\x dt,\\
&I_4= - \int_0^T\int_\Omega (\rho\uu\otimes\uu\cdot \nabla\varphi)\cdot\Delta^{-1} \nabla \Big( \overline{B(\rho)} F([\rho]^\eta)\Big) d\x dt,\\
&I_5=-\int_0^T\int_\Omega \varphi\rho\uu\otimes\uu : \nabla\Delta^{-1} \nabla \Big( \overline{B(\rho)} F([\rho]^\eta)\Big) d\x dt.
\end{align*}

\textit{Step \#2.} We claim that $I_j^n\to I_j$ for $j=2,3,4$ and that $I_1^n+I_5^n\to I_1+I_5$ as $n\to \infty$. Indeed, using the boundedness of the function $B$, another application of Proposition~\ref{A.B.prop2} yields the convergence
\begin{equation}\label{2.BnFntoBF}
B(\rho_n)F(\rhoneta)\to \overline{B(\rho)}F(\rhoeta)\quad \text{in }C([0,T];L_{weak}^p(\Omega)), \quad\text{for any finite }p>1.
\end{equation}
Consequently, we may use the regularizing properties of the operator $\Delta^{-1}\nabla$ (see  Proposition~\ref{A.C.1}) to conclude that
\begin{equation}\label{2.BnFntoBFstrongC}
\Delta^{-1}\nabla\Big( B(\rho_n) F([\rho_n]^\eta)\Big)\to \Delta^{-1}\nabla\Big( \overline{B(\rho)} F([\rho]^\eta)\Big)\quad\text{in }C(K),
\end{equation}
for any compact $K\subseteq [0,T]\times\overline{\Omega}$. Thus, 
\begin{equation}
I_3^n\to I_3\quad\text{and}\quad I_4^n\to I_4^n \text{ as }n\to\infty.
\end{equation}

Similarly, by Lemma~\ref{2.l.eqBF} we have that $h_n\to h$ weakly in $L^2((0,T)\times\Omega)$ and so
\[
\Delta^{-1}\nabla h_n\to \Delta^{-1}\nabla h\quad\text{weakly in }L^2(0,T;H^1(\Omega)),
\]
which, together with \eqref{2.eq.rhonunstrong} implies that
\begin{equation}
I_2^n\to I_2\quad\text{as }n\to \infty.
\end{equation}

Regarding $I_1^n$ and $I_5^n$ we see that
\begin{multline*}
I_1^n+I_5^n = \int_0^T\int_\Omega \uu_n \cdot\Big[ B(\rho_n)F(\rhoneta)\nabla\Delta^{-1}\Div(\varphi\rho_n\uu_n) \\
 - \nabla\Delta^{-1}\nabla\Big(B(\rho_n)F(\rhoneta)\Big)\cdot\varphi\rho_n\uu_n  \Big] d\x dt,
\end{multline*}
and, likewise,
\begin{multline*}
I_1+I_5 = \int_0^T\int_\Omega \uu \cdot\Big[ \overline{B(\rho)}F(\rhoeta)\nabla\Delta^{-1}\Div(\varphi\rho\uu) \\
 - \nabla\Delta^{-1}\nabla\Big(\overline{B(\rho)}F(\rhoeta)\Big)\cdot\varphi\rho\uu  \Big] d\x dt,
\end{multline*}

From the relations \eqref{2.BnFntoBF}, \eqref{2.rhonuntorhouweak} and Theorem~\ref{A.C.3} we have that
\begin{multline*}
B(\rho_n)F(\rhoneta)\nabla\Delta^{-1}\Div(\varphi\rho_n\uu_n)  - \nabla\Delta^{-1}\nabla\Big(B(\rho_n)F(\rhoneta)\Big)\cdot\varphi\rho_n\uu_n \\
\to  \overline{B(\rho)}F(\rhoeta)\nabla\Delta^{-1}\Div(\varphi\rho\uu)  - \nabla\Delta^{-1}\nabla\Big(\overline{B(\rho)}F(\rhoeta)\Big)\cdot\varphi\rho\uu 
\end{multline*}
weakly in $L^r(\Omega)$ for any $1\leq r<\frac{2\gamma}{\gamma+1}$ and for each fixed $t\in[0,T]$. Since $\gamma>N/2$, we have that $L^r(\Omega)$ is compactly imbedded in $H^{-1}(\Omega)$ and thus,
\begin{multline*}
B(\rho_n)F(\rhoneta)\nabla\Delta^{-1}\Div(\varphi\rho_n\uu_n)  - \nabla\Delta^{-1}\nabla\Big(B(\rho_n)F(\rhoneta)\Big)\cdot\varphi\rho_n\uu_n \\
\to  \overline{B(\rho)}F(\rhoeta)\nabla\Delta^{-1}\Div(\varphi\rho\uu)  - \nabla\Delta^{-1}\nabla\Big(\overline{B(\rho)}F(\rhoeta)\Big)\cdot\varphi\rho\uu \\
\text{strongly in }L^2(0,T;H^{-1}(\Omega)).
\end{multline*}
Hence, using \eqref{2.convergen} we conclude that
\begin{equation}
I_1^n+I_5^n\to I_1+I_5\quad\text{as }n\to\infty,
\end{equation}
which proves the claim.

Looking back at identities \eqref{2.idn} and \eqref{2.idlimit} we obtain that
\begin{multline}\label{2.eq.limitEVP}
\lim_{n\to\infty}\int_0^T\int_\Omega P(\rho_n) \Div\phi_n\, d\x dt -\int_0^T\int_\Omega \mathbb{S}_n\cdot \nabla\phi_n \, d\x dt\\
= \int_0^T\int_\Omega \overline{P} \Div\phi\, d\x dt -\int_0^T\int_\Omega \mathbb{S}\cdot \nabla\phi \, d\x dt
\end{multline}

\textit{Step \#3.} Now, we claim that
\begin{equation}\label{2.limitEVP}
\lim_{n\to\infty}\int_0^T\int_\Omega \varphi q_n B(\rho_n) F([\rho_n]^\eta)d\x dt = \int_0^T\int_\Omega \varphi q \overline{B(\rho)} F([\rho]^\eta)d\x dt,
\end{equation}
where $q_n$ and $q$, which are given by \eqref{2.eq.EVP} and \eqref{2.eq.EVPn}, are the effective viscous pressure corresponding to the momentum equations \eqref{2.eq.momentumn} and \eqref{2.eq.momentumlimit}, respectively.

In light of \eqref{2.eq.limitEVP} and Lemma~\ref{2.l.eval}, it suffices to show that
\[
I_j^n\to I_j\quad\text{as }n\to\infty, 
\]
for $j=6,7$, where
\begin{align*}
&I_6^n=- \int_0^T\int_\Omega [\nabla\Delta^{-1}\nabla,\varphi]:(\mathbb{S}_n)B(\rho_n) F([\rho_n]^\eta)d\x dt,\\
&I_7^n= \int_0^T\int_\Omega \Big((P(\rho_n)\ \mathbb{I}-\mathbb{S}_n)\cdot\nabla\varphi\Big)\cdot \Delta^{-1}\nabla \Big(B(\rho_n) F([\rho_n]^\eta)\Big)d\x dt
\end{align*}
and, accordingly,
\begin{align*}
&I_6=- \int_0^T\int_\Omega [\nabla\Delta^{-1}\nabla,\varphi]:(\mathbb{S})\overline{B(\rho)} F([\rho]^\eta)d\x dt\\
&I_7= \int_0^T\int_\Omega \Big((\overline{P}\ \mathbb{I}-\mathbb{S})\cdot\nabla\varphi\Big)\cdot \Delta^{-1}\nabla \Big(B(\rho) F([\rho]^\eta)\Big)d\x dt
\end{align*}

In view of \eqref{2.BnFntoBFstrongC} it is clear that
\[
I_7^n\to I_7\quad\text{as }n\to\infty.
\]

Regarding $I_6^n$ we may invoke the regularizing properties of the commutator of Riesz transforms and the operator of multiplication, discovered by Coifman and Meyer in \cite{CM}, which we state in the Appendix, for convenience (see Theorem~\ref{A.C.2}). Indeed, by part (2) of Theorem~\ref{A.C.2} we have that
\begin{multline*}
[\nabla\Delta^{-1}\nabla,\varphi]:(\mathbb{S}_n) \\
\to [\nabla\Delta^{-1}\nabla,\varphi]:(\mathbb{S})\quad \text{weakly in }L^2(0,T;W^{1,q}(\mathbb{R}^N)), \text{ for any }1<q<2.
\end{multline*}
Therefore, noting that \eqref{2.BnFntoBF} implies that
\begin{multline}\label{2.eq.2.32}
B(\rho_n)F(\rhoneta)\\
\to \overline{B(\rho)}F(\rhoeta)\quad\text{strongly in }C([0,T];W^{-1,p}(\mathbb{R}^N)),\text{ for any }1<p<\infty,
\end{multline}
then, we readily conclude that
\[
I_6^n\to I_6\quad\text{as }n\to\infty.
\]

\textit{Step \#4.} At last, with identity \eqref{2.limitEVP} at hand, from the definition of $q_n$ and $q$, we see that all that is left to conclude the proof of the Theorem is to show that $I_8^n\to I_8$ as $n\to\infty$, where
\[
I_8^n= \int_0^T\int_\Omega\varphi [\nabla\Delta^{-1}\nabla,\mu(\rhoneta)]:\mathbb{D}(\uu_n)\ B(\rho_n)F(\rhoneta)d\x dt,
\]
and
\[
I_8^n= \int_0^T\int_\Omega\varphi [\nabla\Delta^{-1}\nabla,\mu(\rhoeta)]:\mathbb{D}(\uu)\ \overline{B}(\rho)F(\rhoeta)d\x dt.
\]

This, however, is another consequence of part (2) of Theorem~\ref{A.C.2}, which implies that
\begin{multline*}
[\nabla\Delta^{-1}\nabla,\mu(\rhoneta)]:\mathbb{D}(\uu_n)\\
\to [\nabla\Delta^{-1}\nabla,\mu(\rhoeta)]:\mathbb{D}(\uu)\quad\text{weakly in }L^2(0,T;W^{1,q}(\mathbb{R}^N)),\\
 \text{ for any }1<q<2,
\end{multline*}
which together with \eqref{2.eq.2.32} yields the desired convergence and concludes the proof.
\end{proof}

\begin{remark}
The proof of \eqref{2.eq.thmEFV2} may be carried out in the same way using the test functions
 \begin{equation*} 
   \begin{cases}
    \phi_n(t,\xx) = \varphi(t,\xx) \Delta^{-1} \nabla \Big( B(\rho_n) \Big)(t,\xx) \text{ and} \\  
    \phi(t,\xx) = \varphi(t,\xx) \Delta^{-1} \nabla \Big( \overline{B(\rho)}  \Big)(t,\xx),
   \end{cases}
  \end{equation*}
  instead of \eqref{2.eqphi}. In fact, the proof in this case is slightly less complicated in the sense that Lemma~\ref{2.l.eqBF} is not necessary, since the equations satisfied by $B(\rho_n)$ and by $\overline{B(\rho)}$ are much simpler than those corresponding to $B(\rho_n)F(\rhoneta)$ and by $\overline{B(\rho)}F(\rhoeta)$.
\end{remark}

\section{Existence of solutions: Regularized problem}\label{S.regularized}

We now turn our attention to the problem of existence of solutions to the Navier-Stokes equations. To this end, let us consider the following auxiliary system
\begin{align}
&\partial_t \rho + \Div(\rho \uu) = \ve\Delta\rho   \label{2.eq.continuity-ve} \\
&\partial_t(\rho \uu) + \Div (\rho \uu \otimes \uu) + \nabla P = \Div \Sbb + \ve\nabla\uu\cdot\nabla\rho \label{2.eq.momentum-ve},
\end{align}
where, $\ve>0$ and $\delta>0$ are small constants,
\begin{equation}
  P=P^{(\delta)}(\rho) \stackrel{\text{def}}{=} A\rho^\gamma + \delta\, \rho^\beta. \label{2.e2}
\end{equation}
with the new exponent $\beta$ satisfying
\begin{equation}
  \beta > \operatorname{Max}\, \Big\{ 4,\frac{3}{2} N, \gamma \Big\}, \label{2.ebeta}
\end{equation}
  
Note that, aside from the artificial viscosity added to the continuity equation, intended to reguralize the density, two new terms are added to the continuity equation. Namely, an artificial pressure term $\delta \rho^\beta$, which will allow for improved integrability of the density, as well as the term $\ve\nabla\uu\cdot\nabla\rho$ which makes up for the unbalance in the energy of the system caused by the introduction of the viscosity in the continuity equation. This approximation of system \eqref{1.eq.continuity}-\eqref{1.eq.momentum} resembles the one introduced by Feireisl, Novotn\'{y}, and Petzeltov\'{a} in \cite{FNP}, where they study the Navier-Stokes equations with constant viscosity coefficients.

We will find solutions of the Navier-Stokes equations as a limit of the sequence of solutions to the regularized system taking the limit as $\ve\to 0$, leaving $\delta$ fixed first, and then taking the limit when $\delta\to 0$. To that end, we consider equations \eqref{2.eq.continuity-ve}-\eqref{2.eq.momentum-ve} with $\mathbb{S}$ given by \eqref{1.eq.VST} and viscosity coefficients of the form \eqref{1.eq.viscosities1} satisfying \eqref{1.eq.viscosities2}, subject to the following initial and boundary conditions
\begin{align}
&\nabla\rho(t,x)\cdot\nu=0, & x\in\po\Omega, 0<t<T\label{2.eq.boundarydata2-ve}\\
&\uu(t,x)=0, & x\in\po\Omega, 0<t<T,\label{2.eq.boundarydata1-ve}\\
&\rho(0,x)=\rho_0^{(\delta)}(x), \rho(0,x)\uu(0,x)=\m_0^{(\delta)}(x), &x\in \Omega,\label{2.eq.initialconds-ve}
\end{align}
where $\nu$ is the normal vector to $\partial\Omega$ and $\rho_0^{(\delta)}$ and $\m_0^{(\delta)}$ are, respectively, suitable approximations of the initial datum $\rho_0$ and $\m_0$ from the original system. Note that a Neumann boundary condition was added to the density in accordance with the introduction of the viscosity term in the continuity equation.

Existence and uniqueness of solutions can be proven as in \cite{FNP} through a Faedo-Galerkin method. More precisely, first the continuity equation \eqref{2.eq.continuity-ve} is solved globally in terms of the velocity field assuming that the latter is as smooth as needed. Then, for each $n\in\mathbb{N}$ we find a solution $\uu_n$ for the Faedo-Galerkin approximations of the momentum equation, namely equation \eqref{2.eq.momentum-ve} with $\rho=\rho_n$ being the solution of \eqref{2.eq.continuity-ve} in terms of $\uu_n$, satisfied in the weak sense with test functions in the finite-dimensional space $X_n\subseteq L^2(\Omega;\mathbb{R}^d)$ generated by the first $n$ eigenfunctions of the Laplacian in $H_0^1(\Omega)$. The Faedo-Galerkin approximations are solved using Schauder's fixed point theorem, which provides a unique local solution $u_n\in X_n$. Finally, it is shown that each $\uu_n$ may be prolonged globally and that the sequence $(\rho_n,\uu_n)$ has a limit point, which is the desired solution of the regularized system, based on a couple of global a priori estimates. Note that the strong convergence of the densities is straightforward due to the artificial viscosity term, which regularizes the continuity equation. In particular, the explicit dependence of the viscosity coefficients of the fluid on the density poses no extra difficulties at this stage. In summary, we have the following.

\begin{proposition} \label{2.thm}
Fix $\ve > 0$, and assume that initial datum $\rho_0^{(\delta)}$ and $\m_0^{(\delta)}$ satisfy the following properties:
\begin{enumerate}
  \item [(a)] $\rho_0^{(\delta)} \in C^\infty(\overline \Omega)$ and $\m_0^{(\delta)} \in C^\infty(\overline \Omega; \R^N)$;
  \item [(b)] there exists a constant $m > 0$ such that $\rho^{(\delta)}(\xx) > m$ everywhere;
  \item [(c)] $\frac{\partial \rho_0^{(\delta)}}{\partial \nu(\xx)}(\xx) = 0$ on $\partial \Omega$.
\end{enumerate}

Then, there exist a weak solution $(\rho^\ve,\rho^\ve\uu^\ve)$ of \eqref{2.eq.continuity-ve}-\eqref{2.eq.momentum-ve} \eqref{2.eq.boundarydata1-ve}-\eqref{2.eq.initialconds-ve} such that
 $$
 (\rho^{(\ve)}, \uu^{(\ve)}) \in \big(C([0,T];L^2(\Omega)) \cap L^2(0,T;H^1(\Omega))\big) \times L^2(0,T;H_0^1(\Omega;\R^N)).
 $$ 
Moreover, $\rho^{(\varepsilon)}\geq 0$ and the pair satisfies the energy estimates
\begin{multline}\label{2.energyestimate}
  \int_\Omega \Big[ \frac{1}{2} \rho^{(\ve)}(t) |\uu^{(\ve)}(t)|^2 + \frac{A}{\gamma-1} \rho^{(\ve)}(t)^\gamma + \frac{\delta}{\beta-1} \rho^{(\ve)}(t)^\beta \Big] \, d\xx\\
  + \int_0^t\int_\Omega \Big[ \mu(\rhoeta) | \nabla \uu^{(\ve)}|^2 + \Big(\lambda(\rhoeta) + \mu(\rhoeta)) (\Div \uu^{(\ve)})^2  \Big] \, d\xx\, dt' \\
  + \ve \int_0^t\int_\Omega \Big[ A \,\gamma\, (\rho^{(\ve)})^{\gamma - 1} + \delta \, (\rho^{(\ve)})^{\beta - 1} \Big] |\nabla \rho^{(\ve)}|^2 \,d\xx\,dt'  \\
  \leq \int_\Omega \Big[ \frac{1}{2} (\rho_0^{(\delta)})^{-1} |\m_0^{(\delta)}|^2 + \frac{A}{\gamma-1} (\rho_0^{(\delta)})^\gamma + \frac{\delta}{\beta-1} (\rho_0^{(\delta)})^\beta \Big] \, d\xx  \stackrel{\text{def}}{=} E_0^{(\delta)}, 
  \end{multline}
  and
  \begin{multline}\label{2.energyestimate2}
  \int_\Omega |\rho^{(\ve)}(t)|^2\, d\xx  + \ve \int_0^t \int_\Omega |\nabla \rho^{(\ve)} |^2\, d\xx\,dt'\\
  \leq \int_\Omega |\rho_0^{(\delta)}|^2\,d\xx + \int_0^t\int_\Omega \Div \uu^{(\ve)} (\rho^{(\ve)})^2\, d\xx\,dt' 
\end{multline}
for almost every $0 < t < T$. 
\end{proposition}

Since the poof of this result, as outlined above, follows exactly as in \cite{FNP} (cf. \cite{F}) we 
omit 
the details. Note that the assumptions on the initial values $\rho_0^{(\delta)}(\xx)$ and $\m_0^{(\delta)}(\xx)$ may be relaxed. However, we can always approximate functions $\rho_0$ and $\m_0$ in the class given by \eqref{1.eq.initialdata} by smoother initial data (depending on the parameter $\delta$ and independent of $\ve$) satisfying the hypotheses of Proposition~\ref{2.thm}, as will be shown later. 

As an immediate consequence of the energy inequalites \eqref{2.energyestimate}--\eqref{2.energyestimate2}, we conclude the following.

\begin{corollary}\label{2.corollary1}
Let  $(\rho^{(\ve)}, \uu^{(\ve)})$ be the solutions of the regularized system provided by Proposition~\ref{2.thm}. Then, the following estimates hold true uniformly in $\ve > 0$
\begin{align}
  &\sup_{0 < t < T} \int_\Omega \rho^{(\ve)}(t,\xx)^\gamma\,     d\xx \leq \text{(const.)}\,E_0^{(\delta)} \label{2.e6},\\
  &\sup_{0 < t < T} \int_\Omega \delta\, \rho^{(\ve)}(t,\xx)^\beta d\xx \leq \text{(const.)}\,E_0^{(\delta)} \label{2.e7},\\
  &\int_\Omega \rho^{(\varepsilon)}(t,\xx) \, d\xx =  \int_\Omega \rho_0^{(\delta)}(\xx) \, d\xx  \text{ for all $0 \leq t \leq T$}        \label{2.e8}, \\
  &\sup_{0 < t < T} \int_\Omega \rho^{(\ve)}(t,\xx) |\uu^{(\ve)}(t,\xx) |^2\, d\xx  \leq \text{(const.)}\,E_0^{(\delta)} \label{2.e9}, \\
  &\int_0^T \int_\Omega \big[ |\uu^{(\ve)}(t,\xx)|^2 + |\nabla \uu(t,\xx)|^2 \big]\, d\xx dt \leq \text{(const.)}\,E_0^{(\delta)}, \text{and} \label{2.e10} \\
  &\ve \int_0^T\int_\Omega | \nabla \rho^{(\ve)}(t,\xx) |^2\,d\xx dt \leq \text{(const. depend. on $E_0^{(\delta)}$ and $T$)}. \label{2.e11}
\end{align}
\end{corollary} 

Note that \eqref{2.e10} is enabled by \eqref{1.eq.viscosities2} and \eqref{2.e11} is uniform in $\ve$ due to \eqref{2.ebeta}.

Next we are going to find a sequence $\ve_n \to 0$ (leaving $\d$ fixed) so that $\rho^{(\ve_n)}$ and $\uu^{(\ve_n)}$ converge to a solution of the original system, based on the a priori estimates above. However, a further a priori estimate on the densities is needed as we cannot assert that $P(\rho^{(\ve)})$ has a weakly convergent subsequence as $\ve\to 0$. Indeed, with the estimates available so far we only have a bound for the pressure in $L^\infty(0,T;$ $L^1(\Omega))$. Fortunately, the improved integrability estimates for the density from \cite{FNP,F} can be repeated line by line in our present case. The idea is to obtain an $L^1$ estimate for $\rho^{(\ve)}P(\rho^{(\ve)})$ by taking 
$$
\Div^{-1} \Big[\rho - \frac{1}{|\Omega|} \int_\Omega \rho_0(\yy)\,d\yy \Big](t,\xx)
$$
as a test function in the momentum equation \eqref{2.eq.momentum-ve}, where $\Div^{-1} : \{ f \in L^p(\Omega); \int_\Omega f \, d\xx = 0 \} \rightarrow W_0^{1,p}(\Omega;\R^N)$ ($1<p<\infty$) denotes the so-called Bogovskii operator (see \cite{Bg,BoSo}). Note that here we have utilized the conservation of the mass \eqref{2.e8}. We omit the details.

\begin{lemma} \label{2.lemmarho}
  There exists a constant $C = C(\beta, E_0^{(\delta)}, \delta), T) > 0$, independent of $\ve$, such that
  \begin{equation*}
    \int_0^T \int_{\Omega} \Big[ \rho^{(\ve)}(t',\xx)^{\gamma + 1} + \rho^{(\ve)}(t',\xx)^{\beta + 1} \Big]\,d\xx\,dt' \leq C. 
  \end{equation*}
\end{lemma}

\begin{remark}
The similar, although weaker, result
$$
\int_a^b \int_{\omega} \Big[ \rho^{(\ve)}(t',\xx)^{\gamma + 1} + \rho^{(\ve)}(t',\xx)^{\beta + 1} \Big]\,d\xx\,dt' \leq C(\beta, E_0^{(\delta)}, \delta, T, \cO),
$$
where $\cO \subset\subset \Omega$ and $0<a<b<T$, may be proven if one chooses alternatively the test function
$$
\Delta^{-1}\nabla \Big[ \psi  \,. \,(\theta_\kappa \star \rho^{(\ve)}) \Big],
$$
with $\psi \in C_c^\infty(\Omega)$ being a localization function, and $(\theta_\kappa)_{\kappa > 0}$ being mollifiers in $C_c^\infty(\R^N)$ (see \cite{F}, lemma~7.6). The operator $\nabla \Delta^{-1}$ will play an important role in the sequence of this paper, and some of its most crucial properties are discussed in the Appendix~\ref{ApdxC}.
\end{remark}

\section{Vanishing viscosity limit}\label{S.veto0}

Now we move on to the vanishing viscosity limit. With the a priori estimates from the previous Section at hand we can find weakly convergent subsequences as $\ve\to 0$ (leaving $\delta$ fixed) so that the limit functions  $\rho$ and $\uu$ solve the Navier-Stokes equations \eqref{1.eq.continuity}-\eqref{1.eq.momentum} with $P=\overline{P(\rho)}$, where $\overline{P(\rho)}$ is a weak limit of the sequence $\{P(\rho^{(\ve)})\}$. As pointed out before, the main difficulty in this scheme is to show that $\overline{P(\rho)}=P(\rho)$, which turns out to be equivalent to showing strong convergence of the densities. To achieve this, we use a variant of Theorem~\ref{2.thm.EVF}. 

We begin by proving the following proposition, which  follows directly from Corollary~\ref{2.corollary1} and Lemma~\ref{2.lemmarho}.

\begin{proposition} \label{2.thmconvs}
  Fixing both $\rho_0^{(\delta)}$ and $\m_0^{(\delta)}$ and passing to subsequences $\ve_n \rightarrow 0$ if necessary, we may assume that there exist $\rho \in L^\infty(0,T;L^\gamma(\Omega))$, $\overline{P(\rho)} \in L^{\frac{\beta+1}{\beta}}((0,T)\X\Omega)$ and $\uu \in L^2(0,T;H_0^1(\Omega; \R^N))$, such that 
\begin{equation} \label{2.convergences}
\begin{cases}
  \rho^{(\ve)}                              \rightharpoonup \rho               &\text{ weakly in } L^{\beta+1}((0,T)\times\Omega), \\
  P(\rho^{(\ve)})                           \rightharpoonup \overline{P(\rho)} &\text{ weakly in } L^{\frac{\beta+1}{\beta}}((0,T)\times\Omega), \\
  \uu^{(\ve)}                               \rightharpoonup \uu                &\text{ weakly in } L^2(0,T;H_0^1(\Omega; \R^N)), \\
  \ve \nabla \uu^{(\ve)}\nabla \rho^{(\ve)} \rightarrow 0                      &\text{ strongly in } L^1((0,T)\times\Omega; \R^N), \text{ and} \\
  \ve \Delta \rho^{(\ve)}                   \rightarrow 0                      &\text{ strongly in $L^2(0,T; H^{-1}(\Omega))$}.
\end{cases}
\end{equation}

Moreover, we have
\begin{enumerate}
  \item[(i)]  $\rho^{(\ve)} \rightarrow \rho$ in $C([0,T];L_{\text{weak}}^{\gamma}(\Omega))$;
  \item[(ii)] $\lambda(\rhoveta) \rightarrow \lambda(\rhoeta)$ and $\mu(\rhoveta) \rightarrow \mu(\rhoeta)$ strongly in $C([0,T];C^\infty(\overline{\Omega}))$ 
  \item[(iii)] $\rho^{(\ve)}\uu^{(\ve)} \rightarrow \rho \uu$ in $C([0,T];L_{\text{weak}}^{\frac{2\gamma}{\gamma + 1}}(\Omega))$;
  \item[(iv)] $\rho^{(\ve)} \uu^{(\ve)}\otimes\uu^{(\ve)} \rightharpoonup \rho\uu\otimes\uu$ weakly in $L^{2}(0,T; L^s(\Omega;\mathbb{R}^{N\times N}))$ for $1 < s < \infty$ satisfying $\frac{1}{s} = \frac{\gamma+1}{2\gamma} + \frac{1}{2^*}$ if $N\geq 3$, and $\frac{1}{s} > \frac{\gamma+1}{2\gamma}$ if $N = 2$.
\end{enumerate}

Finally, the pair $(\rho,\uu)$ is a weak solution to
  \begin{align} \label{2.system2}
&\partial_t \rho + \Div(\rho \uu) = 0, \\    
    &\partial_t \big(\rho \uu\big) + \Div \big(\rho \uu \otimes \uu \big) + \nabla \overline{P(\rho)} = \Div \Sbb,
\end{align}
with initial and boundary conditions
\begin{align}
    &\uu(t,x) = 0 &\text{ for $0 < t < T$ and  $\xx \in \partial\Omega$, }\label{2.system2conds0} \\
    &\rho(0,x) = \rho^{(\delta)}_0(x),\quad (\rho \uu)(0,x) = \m^{(\delta)}_0(x) &\text{ for }\xx \in \Omega,\label{2.system2conds}
  \end{align}
where $\Sbb = \lambda(\rhoeta)\, (\Div \uu) \mathbb{I} + 2\mu(\rhoeta)\, \mathbb{D}(\uu)$, and satisfies the energy estimate
\begin{multline}\label{2.energy}
  \int_\Omega \Big[ \frac{1}{2} \rho(t) |\uu(t)|^2 + \frac{A}{\gamma-1} \rho(t)^\gamma + \frac{\delta}{\beta-1} \rho(t)^\beta \Big] \, d\xx \\
  + \int_0^t\int_\Omega \Big[ \mu(\rhoeta) | \nabla \uu|^2 + \Big(\lambda(\rhoeta) + \mu(\rhoeta)\Big) (\Div \uu)^2  \Big] \, d\xx\, dt'  \\
  \leq \int_\Omega \Big[ \frac{1}{2} (\rho_0^{(\delta)})^{-1} |\m_0^{(\delta)}|^2 + \frac{A}{\gamma-1} (\rho_0^{(\delta)})^\gamma + \frac{\delta}{\beta-1} (\rho_0)^{(\delta)})^\beta \Big] \, d\xx = E_0^{(\delta)}
\end{multline}
for almost every $0 < t < T$.
\end{proposition}

\begin{remark}
Let us mention that, by weak solution to \eqref{2.system2}, it should be understood that $(\rho, \uu)$ verifies
  \begin{align}
    -\int_0^T\int_\Omega &\rho \uu \cdot \frac{\partial \varphi}{\partial t}\, d\xx\,dt - \int_0^T \int_\Omega \rho \uu \otimes \uu : \nabla \varphi\, d\xx dt  - \int_0^T \int_\Omega \overline{P(\rho)} \Div \varphi\, d\xx dt \nonumber \\
    &\quad\quad\quad\quad\quad\quad + \int_0^T\int_\Omega \mathbb{S}: \nabla \varphi\,d\xx dt  =  \int_\Omega \m_0^{(\delta)}(\xx). \varphi(0,\xx)\, d\xx\text{, and} \label{2.eqm} \\
    \int_0^T\int_\Omega &\rho \phi\, d\xx dt - \int_0^T \int_\Omega \rho \uu . \nabla \phi\, d\xx dt = \int_\Omega \rho_0^{(\delta)}(\xx) \phi(0,\xx)\, d\xx \label{2.eqrho}
  \end{align}
  for any $\varphi \in C_c^\infty([0,T)\X\Omega;\mathbb{R}^N)$ and $\phi \in C^\infty([0,T)\X\overline\Omega)$. 
\end{remark}

\begin{proof}
First we observe that the convergences in \eqref{2.convergences} follow directly from Corollary~\ref{2.corollary1}. Second, note that (i) follows from Proposition~\ref{A.B.prop2} in view of the equations verified by $\rho^{(\ve)}$. Next, (ii) follows from the fact that the convolution operator $f \in L^\beta(\Omega) \mapsto (\eta \star \widetilde f) \in C^k(\overline\Omega)$ is compact for any integer $k \geq 0$, as one can easily see from the Arzelá--Ascoli theorem. 

The assertion that $\rho^{(\ve)} \uu^{(\ve)}$ converges in $C([0,T];L^{\frac{2\gamma}{\gamma+1}}(\Omega))$ follows from the same lines of (i) and the bounds of $\rho^{(\ve)}$ in $L^\infty(0,T;L^\gamma(\Omega))$ and of $\sqrt{\rho^{(\ve)}} \uu^{(\ve)}$ in $L^\infty(0,T;L^2(\Omega))$  expressed in, respectively, \eqref{2.e6} and \eqref{2.e7}. Then, in order to conclude (iii) we have to verify that the limit is indeed $\rho \uu$ (symbolically, $\overline{\rho \uu} = \overline{\rho}\,\overline{\uu}$). This can be seen as follows. Since $\gamma > N/2$, $L^\gamma(\Omega) \subset H^{-1}(\Omega)$ with compact injection. Then $\rho^{(\ve)} \rightarrow \rho$ strongly in $C([0,T]; H^{-1}(\Omega))$. This implies that $\overline{\rho \uu} = \rho \uu$ indeed.

Finally, (iv) is obtained by a similar argument of (iii), for $L^{\frac{2\gamma}{\gamma+1}}(\Omega) \subset H^{-1}(\Omega)$ compactly since $\gamma > N/2$.

All things considered, one can easily pass to the limit and conclude that \eqref{2.eqm} and \eqref{2.eqrho} are both valid, i.e., that $(\rho, \uu)$ is a equation solution to \eqref{2.system2}-\eqref{2.system2conds}. Moreover,  \eqref{2.energyestimate} implies that
\begin{multline*}
  -\int_0^T\int_\Omega \Big[ \frac{1}{2} \rho^{(\ve)} |\uu^{(\ve)}|^2 + \frac{A}{\gamma-1} (\rho^{(\ve)})^\gamma + \frac{\delta}{\beta-1} (\rho^{(\ve)})^\beta \Big] \psi'(t) \, d\xx dt \\
  + \int_0^T \int_\Omega \Big[ \mu(\rhoeta) | \nabla \uu^{(\ve)}|^2 + \Big(\lambda(\rhoeta) + \mu(\rhoeta)\Big) (\Div \uu^{(\ve)})^2  \Big] \psi(t) \, d\xx\, dt \nonumber \\
  \leq \int_\Omega \Big[ \frac{1}{2} \rho_0^{(\delta)} |\uu_0^{(\delta)}(\xx)|^2 + \frac{A}{\gamma-1} (\rho_0^{(\delta)})^\gamma + \frac{\delta}{\beta-1} (\rho_0^{(\delta)})^\beta \Big] \psi(0) \, d\xx  = \psi(0) E_0^{(\delta)},
\end{multline*}
for any nonnegative $\psi \in C_c^\infty([0,T))$, from which the energy inequality \eqref{2.energy} follows.
 \end{proof}

Next we state a variant of Theorem~\ref{2.thm.EVF} valid for solutions of the regularized system.

\begin{lemma} \label{3.thm.effviscflux}
    For any $\varphi \in C_c^\infty((0,T)\times\Omega)$,
    \begin{multline*}
      \int_0^T \int_\Omega \varphi \rho^{(\ve)} \Big[ \frac{P(\rho^{(\ve)})}{2\mu(\rhoveta) + \lambda(\rhoveta)} - \Div \uu^{(\ve)} \Big] \, d\xx dt  \\
               \rightarrow \int_0^T \int_\Omega \varphi \rho \Big[ \frac{\overline{P(\rho)}}{2\mu(\rhoeta) + \lambda(\rhoeta)} - \Div \uu \Big] \, d\xx dt.
    \end{multline*}
  \end{lemma}

The proof is essentially the same as that of Theorem~\ref{2.thm.EVF}, modulo a few terms that tend to zero as $\ve\to 0$. Note that, in contrast with Theorem~\ref{2.thm.EVF}, here the conclusion holds with the function $B(z)=z$, which is not a bounded function. However, due to the artificial pressure term, which is fixed throughout this Section, there are higher integrability estimates available on the densities which allow for the admissibility of this unbounded function.

Since the main ideas of the proof have been set in Section~\ref{S.EVF} we will only point out the modifications that are in order to prove this result. 

First, we have the following observation which is the analogue of Lemma~\ref{2.l.eqBF} for solutions of the regularized system.

\begin{lemma}\label{4.l.eqrhoF}
Denote
 \begin{equation*}
    F(\xi) = \frac{1}{\lambda(\xi) + 2\mu(\xi)}. 
  \end{equation*}
Then, the following equation holds in the sense of distributions in $\mathbb{R}^N$
\begin{equation}\label{3.e.rhoveFve}
 \Big(\rho^{(\ve)}F(\rhoveta)\Big)_t + \Div \Big(\rho^{(\ve)} F(\rhoveta )  \uu^{(\ve)} \Big)  = h^{(\ve)},
\end{equation}
where
\begin{multline*}
h^{(\ve)}\stackrel{\text{def}}{=}\ve F(\rhoveta) \,\Div ( 1_\Omega \nabla \rho^{(\ve)})  + \ve \sum_{i=1}^N \rho^{(\ve)}F'(\rhoveta) \Big( \frac{\partial \eta}{\partial y_i} \star \frac{\partial \rho^{(\ve)}}{\partial y_i} \Big) \\ 
       + \rho^{(\ve)} F'(\rhoveta)\Big[\Div \big(\rhoveta  \uu^{(\ve)} \big) - \Div \big(\eta \star (\rho^{(\ve)} \uu^{(\ve)})\big) \Big] \\
       - \rhoveta F'(\rhoveta) \Div \uu^{(\ve)}
\end{multline*}
  
Moreover, $h^{(\ve)}\to h$ weakly in $L^2(0,T;H^{-1}(\mathbb{R}^N)) + L^{\frac{2\beta}{\beta + 2}}((0,T) \times \R^N)$, as $\ve\to 0$, where
\[
    h \stackrel{\text{def}}{=} F'(\rhoeta)\Big[ \rho \nabla\rhoeta  \cdot \uu + \rhoeta\, \overline{\rho \Div \uu} - \Div \big(\eta \star (\rho \uu )\big) \Big]   - \rhoeta\, F'(\rhoeta),
  \]
and $(\rho, \uu)$ satisfy the following equation in the sense of distributions
\begin{equation}\label{3.e.continuityRN}
    \big(\rho F(\rhoeta)\big)_t + \Div \big(\rho F(\rhoeta )  \uu \big)  =\, h,
 \end{equation}
\end{lemma}

\begin{proof}
Note that, extending $\rho^{(\ve)}$ and $\uu^{\ve}$ by zero outside $\Omega$, we have that they satisfy the following equation in the sense of distributions in $\mathbb{R}^N$
\begin{equation}\label{3.e.continuityveRN}
\partial_t\rho^{(\ve)} + \Div(\rho^{(\ve)}\uu^{(\ve)}) = \Div(1_\Omega \nabla \rho^{(\ve)}).
\end{equation}

Then, in order to deduce \eqref{3.e.rhoveFve} first we take $\eta(\x-\y)$ as a test function in \eqref{3.e.continuityveRN} in order to deduce an equation for $\rhoveta$. Then, we use this equation to deduce an equation for $F(\rhoveta)$. Finally, we use the resulting equation and combine it with equation \eqref{3.e.continuityveRN} to conclude. We omit the details.

In order to deduce equation \eqref{3.e.continuityRN} it suffices to take the limit as $\ve\to 0$ in equation \eqref{3.e.rhoveFve}, noting that each term converges weakly to its  counterpart in light of Proposition \ref{2.thmconvs}.
\end{proof}
  
  \begin{proof}[Proof of Lemma~\ref{3.thm.effviscflux}]
  The idea is to take the test functions
  \begin{equation*} 
   \begin{cases}
    \phi^{(\ve)}(t,\xx) = \varphi(t,\xx) \Delta^{-1} \nabla \Big( \frac{\rho^{(\ve)}}{\lambda(\rhoveta) + 2\mu(\rhoveta)} \Big)(t,\xx), \\  
    \phi(t,\xx) = \varphi(t,\xx) \Delta^{-1} \nabla \Big( \frac{\rho}{\lambda(\rhoeta)+ 2\mu(\rhoeta)} \Big)(t,\xx)
   \end{cases}
  \end{equation*}
  in the momentum equations of, respectively, $\rho^{(\ve)}\uu^{(\ve)}$ and $\rho \uu$. In light of Lemma~\ref{4.l.eqrhoF}, we may proceed as in the proof of Theorem~\ref{2.thm.EVF} to conclude that, after some manipulation, we arrive at the following identity
\begin{equation}
\int_0^T \int_\Omega \varphi \rho^{(\ve)} \Big[ \frac{P(\rho^{(\ve)})}{2\mu(\rhoveta) + \lambda(\rhoveta)} - \Div \uu^{(\ve)} \Big] \, d\xx dt = R^{(\ve)}+\sum_{j=1}^{8} I_j^{(\ve)},
\end{equation}
where
\begin{align*}
&R^{(\ve)}= \ve \int_0^T \int_\Omega (\nabla \uu^{(\ve)}\cdot \nabla \rho^{(\ve)}) \cdot  \Delta^{-1} \nabla \Big[ \rho^{(\ve)} F(\rhoveta) \Big]\varphi \, d\xx dt \\
&I_1^{(\ve)}=\int_0^T \int_\Omega \varphi\rho^{(\ve)}\uu^{(\ve)} \Delta^{-1} \nabla\Div \Big( B(\rho^{(\ve)}) F([\rho^{(\ve)}]^\eta)\uu^{(\ve)} \Big)d\x dt,\\
&I_2^{(\ve)}=-\int_0^T \int_\Omega \varphi\rho^{(\ve)}\uu^{(\ve)} \Delta^{-1} \nabla h^{(\ve)} d\x dt,\\
&I_3^{(\ve)} = -\int_0^T \int_\Omega \rho^{(\ve)}\uu^{(\ve)} \varphi_t\Delta^{-1} \nabla \Big( B(\rho^{(\ve)}) F([\rho^{(\ve)}]^\eta)\Big) d\x dt,\\
&I_4^{(\ve)}= - \int_0^T\int_\Omega (\rho^{(\ve)}\uu^{(\ve)}\otimes\uu^{(\ve)}\cdot \nabla\varphi)\cdot\Delta^{-1} \nabla \Big( B(\rho^{(\ve)}) F([\rho^{(\ve)}]^\eta)\Big) d\x dt,\\
&I_5^{(\ve)}=-\int_0^T\int_\Omega \varphi\rho^{(\ve)}\uu^{(\ve)}\otimes\uu^{(\ve)} : \nabla\Delta^{-1} \nabla \Big( B(\rho^{(\ve)}) F([\rho^{(\ve)}]^\eta)\Big) d\x dt,\\
&I_6^{(\ve)}= \int_0^T\int_\Omega [\nabla\Delta^{-1}\nabla,\varphi]:(\mathbb{S}^{(\ve)})B(\rho^{(\ve)}) F([\rho^{(\ve)}]^\eta)d\x dt,\\
&I_7^{(\ve)}= -\int_0^T\int_\Omega \Big((P(\rho^{(\ve)})\ \mathbb{I}-\mathbb{S}^{(\ve)})\cdot\nabla\varphi\Big)\cdot \Delta^{-1}\nabla \Big(B(\rho^{(\ve)}) F([\rho^{(\ve)}]^\eta)\Big)d\x dt,\\
&I_8^{(\ve)}= 2\int_0^T\int_\Omega\varphi [\nabla\Delta^{-1}\nabla,\mu(\rhoveta)]:\mathbb{D}(\uu^{(\ve)})\ B(\rho^{(\ve)})F(\rhoveta)d\x dt,
\end{align*}

Similarly, after taking $\phi^{(\ve)}$ as a test function in equation \eqref{2.system2}, using Lemma~\ref{4.l.eqrhoF} and proceeding as in the proof of Theorem~\ref{2.thm.EVF} we arrive at the identity
\begin{equation*}
\int_0^T \int_\Omega \varphi \rho \Big[ \frac{\overline{P(\rho)}}{2\mu(\rhoeta) + \lambda(\rhoeta)} - \Div \uu \Big] \, d\xx dt = \sum_{j=1}^{8} I_j^{(\ve)},
\end{equation*}
where
\begin{align*}
&I_1=\int_0^T \int_\Omega \varphi\rho\uu\ \Delta^{-1} \nabla\Div \Big( \rho F([\rho]^\eta)\uu \Big)d\x dt,\\
&I_2=-\int_0^T \int_\Omega \varphi\rho\uu \Delta^{-1} \nabla h d\x dt,\\
&I_3 = -\int_0^T \int_\Omega \rho\uu\ \varphi_t\  \Delta^{-1} \nabla \Big( \rho F([\rho]^\eta)\Big) d\x dt,\\
&I_4= - \int_0^T\int_\Omega (\rho\uu\otimes\uu\cdot \nabla\varphi)\cdot\Delta^{-1} \nabla \Big( \rho F([\rho]^\eta)\Big) d\x dt,\\
&I_5=-\int_0^T\int_\Omega \varphi\rho\uu\otimes\uu : \nabla\Delta^{-1} \nabla \Big(\rho F([\rho]^\eta)\Big) d\x dt,\\
&I_6= \int_0^T\int_\Omega [\nabla\Delta^{-1}\nabla,\varphi]:(\mathbb{S})\ \rho F([\rho]^\eta)d\x dt,\\
&I_7= -\int_0^T\int_\Omega \Big((\overline{P(\rho)}\ \mathbb{I}-\mathbb{S})\cdot\nabla\varphi\Big)\cdot \Delta^{-1}\nabla \Big(\rho F([\rho]^\eta)\Big)d\x dt,\\
&I_8= 2\int_0^T\int_\Omega\varphi [\nabla\Delta^{-1}\nabla,\mu(\rhoeta)]:\mathbb{D}(\uu)\ \rho F(\rhoeta)d\x dt,
\end{align*}

Now, other than the fact that $R^{(\ve)}$ tends to zero as $\ve\to 0$, which follows directly from Proposition~\ref{2.thmconvs}, the proof of Theorem~\ref{2.thm.EVF} can be repeated line by line with some minor modifications in order to show that (up to a subsequence)
\[
I_j^{\ve}\to I_j, \quad\text{as }\ve\to 0,\text{ for all }i=1,...,8,
\] 
and the result follows. We omit the details.
\end{proof}

We are almost in condition to prove the strong convergence of the densities. However, first we need the following Lemma, which follows directly from a general known result (see proposition 4.2 from \cite{F}).

\begin{lemma} \label{2.renormalized}
 If we prolong $\rho$ and $\uu$ to be zero outside $\Omega$, then $\rho$ turns out to be a renormalized solution to the continuity equation $\rho_t + \Div(\rho\uu) = 0$. Moreover, the class of $B$ for which
 \begin{align*}
   B(\rho)_t + \Div (B(\rho)\uu) + (B'(\rho)\rho &- B(\rho))\Div \uu = 0 \\&\text{ in the sense of distributions in $(0,T)\X\R^N$}
 \end{align*}
 can be extended to $B \in C[0,\infty) \cap C^1(0,\infty)$ satisfying
 \begin{equation}
   |\xi B'(\xi)| \leq \text{(const.)}(\xi^{\theta} + \xi^{\gamma/2}) \text{ for all $\xi > 0$}, \label{2.classB}
 \end{equation}
 for some fixed exponent $0 < \theta < \gamma/2$. 
 \end{lemma}

This is basically due to the fact that $\rho \in L^2((0,T)\X \Omega)$ and $\uu \in L^2(0,T;$ $ H_0^1(\Omega))$. The proof consists in mollifying equation the continuity equation, multiplying the resulting equation by $B'(\rho)$ and then taking the limit as the regularizing parameter vanishes, wherein the convergence is ensured by the $L^2$ integrability of the density, which is available at this stage because of the estimates available due to the artificial pressure. 

\begin{proposition}
    $\rho^{(\ve)} \rightarrow \rho$ strongly in $L^1((0,T)\X\Omega)$. 
  \end{proposition}
  
  \begin{proof}
  Since $\rho^{(\ve)}$ is uniformly bounded in $L^{\beta+1}((0,T)\X\Omega)$, it suffices to show that, passing to a subsequence if necessary, $\rho^{(\ve)} \rightarrow \rho$ almost everywhere. To prove this, we will apply a convexity argument; more specifically, we will show that $\overline{(\rho \log \rho)}(t,\xx) = (\rho \log \rho)(t,\xx)$ for almost every $0<t<T$ and $\xx \in \Omega$.
  
  Let $B(z) = z \log z$. On the one hand, in light of Lemma~\ref{2.renormalized}, the function $B(\rho)$ is admissible in the definition of renormalized solutions and therefore we have that the following equation is satisfied in the sense of distributions in $\mathbb{R}^N$
  \begin{equation*}
  \partial_t(\rho\log\rho) +\Div(\rho\log(\rho)\uu)+\rho\Div\uu =0.
  \end{equation*}
Thus, we see that
\begin{equation}\label{4.eq.rhologrho}
\int_0^t\int_\Omega \rho\Div\uu d\x ds = \int_\Omega\rho_0^{(\d)}\log(\rho_0^{\d})d\x - \int_\Omega \rho(t,\x)\log(\rho(t,x))d\x,
\end{equation} 
for any $t\in [0,T]$.

Next, we approximate $B$ by a sequence of smooth convex functions $B_k$, $k=1,2,...$ with $B'$ and $B''$ uniformly bounded and multiply equation \eqref{2.eq.continuity-ve} by $B_k'(\rho^{(\ve)})$ to obtain that, in the sense of distributions in $\mathbb{R}^N$,
\begin{multline*}
\partial_t(B_k(\rho^{(\ve)})) +\Div(B_k(\rho^{(\ve)})\uu^{(\ve)})+(B_k'(\rho^{(\ve)})\rho^{(\ve)}-B_k(\rho^{(\ve)}))\Div\uu^{(\ve)} \\
= \ve\Delta(1_{\Omega}B_k(\rho^{(\ve)}))-\ve 1_\Omega B_k''(\rho^{(\ve)})|\nabla\rho^{(\ve)}|^2.
\end{multline*}

Thus, multiplying this equation by some test function $\psi$, integrating and sending $k\to\infty$, we obtain
\begin{equation*}
\int_0^t\int_\Omega \psi \rho^{(\ve)}\Div\uu^{(\ve)} d\x ds \leq  \int_\Omega \psi \rho_0^{(\d)}\log(\rho_0^{\d})d\x - \int_\Omega \psi \rho^{(\ve)}(t,\x)\log(\rho^{(\ve)}(t,\x))d\x.
\end{equation*} 
for any $t\in [0,T]$.

Passing to a subsequence if necessary, taking the limit as $\ve\to 0$ we get
\begin{equation*}
\int_0^t\int_\Omega \psi \overline{\rho\Div\uu} d\x ds \leq  \int_\Omega \psi \rho_0^{(\d)}\log(\rho_0^{\d})d\x - \int_\Omega \psi \overline{\rho\log(\rho)}(t)d\x,
\end{equation*} 
which implies
\begin{equation}\label{4.eq.rhologrholim}
\int_0^t \int_\Omega \overline{\rho\Div\uu} d\x ds \leq  \int_\Omega  \rho_0^{(\d)}\log(\rho_0^{\d})d\x - \int_\Omega  \overline{\rho\log(\rho)}(t)d\x.
\end{equation} 
  
 In this way, from equations \eqref{4.eq.rhologrho} and \eqref{4.eq.rhologrholim} we infer that
\begin{equation*}
\int_\Omega (\overline{\rho\log(\rho)} -\rho\log\rho)(t) d\x \leq \int_0^t\int_\Omega \rho\Div\uu - \overline{\rho\Div\uu}d\x ds,
\end{equation*}
  for almost all $t\in[0,T]$.
  
  Therefore, by Lemma~\ref{3.thm.effviscflux} we have that
  $$\int_0^T \int_\Omega \big[ \rho \log \rho - \overline{\rho \log \rho } \big]\,d\xx dt \leq \int_0^T \int_\Omega \Big[ \frac{\overline{\rho P(\rho)} - \rho\, \overline{P(\rho)}}{2\mu(\rhoeta) + \lambda(\rhoeta)} \Big] \,d\xx dt. $$
  As, by convexity, 
  \begin{equation}
\rho\, \overline{P(\rho)} \leq \overline{\rho P(\rho)}, \label{2.convexity}
  \end{equation}
	 we deduce that $\rho \log \rho \leq \overline{\rho \log \rho}$. This last bit of information is enough to conclude the strong convergence of the densities due to Proposition~\ref{teo2.11}.
  \end{proof}

\begin{remark}
To see \eqref{2.convexity}, consider a measurable set $X \subset (0,T)\X\Omega$ and define the functional $F: L^{\beta+1}((0,T)\X\Omega) \rightarrow \R$ by $$F(f) = \int_X f((A |f|^{\gamma -1}f + \delta |f|^{\beta-1} f) - \overline{P(\rho)}) \, d\xx dt.$$ It is clear that $F$ is convex and continuous in the strong topology of $L^{\beta+1}(\Omega)$; it is consequently lower semicontinuous in the weak topology $\sigma(L^{\beta+1}, L^{(\beta+1)'})$. As $\rho^{(\ve)} \rightharpoonup \rho$ weakly, we conclude that $$F(\rho) \leq \liminf F(\rho^{(\ve)}),$$ which, passing to a subsequence if necessary, yields that $$0 \leq \int_X \overline{\rho P(\rho)} \, d\xx - \int_X \rho \overline{P(\rho)} \, d\xx dt.$$ This evidently yields that $\rho \overline{P(\rho)} \leq \overline{\rho P(\rho)}$ almost everywhere, as we previously claimed.
  \end{remark}

All things considered, we conclude that the limit functions $(\rho\uu)$ are a weak energy solution of the Navier-Stokes equations. Thus, we have proven the following.

\begin{theorem}\label{4.thm1}
Let $\beta>\max\{ 4,\frac{3}{2}N,\gamma \}$ and let $\rho_0^{(\d)}$ and $\mathbf{m}_0^{(\d)}$ be as in Proposition~\ref{2.thm}. Assume that the viscosity coefficients satisfy \eqref{1.eq.viscosities1}-\eqref{1.eq.viscosities2}. Then, there exists a weak energy solution $(\rho,\rho\uu)$ of equations \eqref{1.eq.continuity}-\eqref{1.eq.momentum} with 
\[
P(\rho)= A\rho^{\gamma}+\d \rho^{\beta},
\]
satisfying the initial and boundary conditions \eqref{1.eq.boundarydata}-\eqref{1.eq.initialconds}.

Moreover, $\rho$ and $\uu$ satisfy the continuity equation in the sense of renormalized solutions and satisfy the estimate
\begin{multline}\label{4.energyestimate}
  \int_\Omega \Big[ \frac{1}{2} \rho(t) |\uu(t)|^2 + \frac{A}{\gamma-1} \rho(t)^\gamma + \frac{\delta}{\beta-1} \rho(t)^\beta \Big] \, d\xx \\
  + \int_0^t\int_\Omega \Big[ \mu(\rhoeta) | \nabla \uu|^2 + \Big(\lambda(\rhoeta) + \mu(\rhoeta)\Big) (\Div \uu)^2  \Big] \, d\xx\, dt' \\
  \leq \int_\Omega \Big[ \frac{1}{2} (\rho_0^{(\delta)})^{-1} |\m_0^{(\delta)}|^2 + \frac{A}{\gamma-1} (\rho_0^{(\delta)})^\gamma + \frac{\delta}{\beta-1} (\rho_0^{(\delta)})^\beta \Big] \, d\xx \stackrel{\text{def}}{=} E_0^{(\delta)}. 
  \end{multline}
\end{theorem}

\section{Vanishing artificial pressure}\label{S.dto0}

Our goal now is to take the limit as $\d\to 0$ and show that, up to a subsequence, $(\rho^{(\d)},\uu^{(\d)})$ converge to a solution for the original problem, where $(\rho^{(\d)},\uu^{(\d)})$ are the solutions of the Navier-Stokes equations provided by Theorem~\ref{4.thm1}. 

Let us consider general datum $\rho_0$ and $\m_0$ as in \eqref{1.eq.initialdata}. Despite they may not fall in the class of initial conditions considered in Theorem~\ref{4.thm1}, we may always approximate them by $\rho_0^{(\delta)} \in C^\infty(\overline \Omega)$, with $\frac{\partial \rho^{(\delta)}}{\partial \nu(\xx)}(\xx)$ along $\xx \in \partial\Omega$, and $\m_0^{(\delta)} \in C^\infty(\overline{\Omega};\R^N)$, for each $\delta > 0$, such that
  \begin{enumerate}
   \item [(i)]   $\rho_0^{(\delta)} \rightarrow \rho_0$ almost everywhere and in $L^\gamma(\Omega)$,
   \item [(ii)]  $0 < \delta \leq \rho_0^{(\delta)} \leq \delta^{-1/(2\beta)}$ everywhere, 
   \item [(iii)] $\m_0^{(\delta)}$ being a suitable regularization by convolution of
   $$\xx \mapsto \sqrt{\rho_0^{(\delta)}(\xx)}\frac{\m_0(\xx)}{\sqrt{\rho_0(\xx)}}$$
   for which $\m_0^{(\delta)}(\rho_0^{(\delta)})^{-1/2} \rightarrow \m_0(\rho_0)^{-1/2}$ in $L^2(\Omega;\R^N)$.
 \end{enumerate}
 This choice is also convenient since it implies that
 $$E_0^{(\delta)} \rightarrow E_0$$
 where $E_0$ and $E_0^{(\delta)}$ are defined in, respectively, \eqref{1.energy} and \eqref{4.energyestimate}.
  
    Let us begin with the following analogue of Lemma \ref{2.lemmarho}.
  
  \begin{lemma} \label{3.lemmarho}
    For any $0<\omega< \operatorname{Min} \{ 1/N, 2\gamma/N - 1 \}$, there exists a constant $C = C(\omega, E_0^{(\delta)})$ such that, for all $\delta > 0,$
    $$\int_0^T\int_\Omega \Big\{ A \, |\rho^{(\delta)}(t,\xx)|^{\gamma + \omega} + \delta \, |\rho^{(\delta)}(t,\xx)|^{\beta + \omega}  \Big\}\,d\xx dt \leq C.$$ 
  \end{lemma}
  
  The proof of this estimate on the improved integrability of the densities, like that of Lemma~\ref{2.lemmarho} may be carried out line by line as the proof of proposition~5.1 in \cite{F} and, therefore, we omit it. The idea is to take the following test function in the momentum equation
  $$\varphi(t,\xx) = \Div^{-1} \Big\{ \big(\theta_n \star B_m(\rho^{(\delta)})\big) - \frac{1}{|\Omega|} \int_\Omega \big(\theta_n \star B_m(\rho^{(\delta)})\big)\,d\yy \Big\}(t,\xx),$$
  where $(\theta_n)$ is a mollifier sequence in $C_c^\infty(\R^N)$ and $B_m(z)$ is an adequate approximation of $z \mapsto z^\omega$. Note that the assumption that $\gamma>\frac{N}{2}$ is essential for $\omega$ to be positive.
 
As a consequence of this estimate and \eqref{4.energyestimate} we deduce the following result whose proof follows the same lines of that of Proposition~\ref{2.thmconvs}.
 
 \begin{proposition}  
Keeping the notations above and passing to subsequences $\delta_n \rightarrow 0$ if necessary, there exist $\rho \in L^\infty(0,T;L^\gamma(\Omega))$, $\overline{P(\rho)} \in L^{1+\frac{\omega}{\gamma}}((0,T)\X\Omega)$ (where $0<\omega<\operatorname{Min} \{ 1/N, 2\gamma/N -1 \}$), and $\uu \in L^2(0,T;H_0^1(\Omega; \R^N))$, such that 
\begin{equation*} 
\begin{cases}
  \rho^{(\delta)}                                 \rightharpoonup \rho               &\text{ weakly in } L^{\gamma+\omega}((0,T)\times\Omega), \\
  P(\rho^{(\delta)})                              \rightharpoonup \overline{P(\rho)} &\text{ weakly in } L^{\frac{\gamma+\omega}{\gamma}}((0,T)\times\Omega),  \text{ and}\\
  \uu^{(\delta)}                                  \rightharpoonup \uu                &\text{ weakly in } L^2(0,T;H_0^1(\Omega; \R^N)).
\end{cases}
\end{equation*}

Moreover, we have the next convergences in the following spaces:
\begin{enumerate}
  \item[(i)]  $\rho^{(\delta)} \rightarrow \rho$ in $C([0,T];L_{\text{weak}}^{\gamma}(\Omega))$;
  \item[(ii)] $\lambda(\rhodeta) \rightarrow \lambda(\rhoeta)$ and $\mu(\rhodeta) \rightarrow \mu(\rhoeta)$ strongly in $C([0,T];C^\infty(\overline{\Omega}))$ (see Remark~\ref{A.A.remark1});
  \item[(iii)] $\rho^{(\delta)}\uu^{(\delta)} \rightarrow \rho \uu$ in $C([0,T];L_{\text{weak}}^{\frac{2\gamma}{\gamma + 1}}(\Omega))$;
  \item[(iv)] $\rho^{(\delta)} \uu^{(\delta)}\otimes\uu^{(\delta)} \rightharpoonup \rho\uu\otimes\uu$ weakly in $L^{2}(0,T; L^s(\Omega;\mathbb{R}^{N\times N}))$ for $1 < s < \infty$ satisfying $\frac{1}{s} = \frac{\gamma+1}{2\gamma} + \frac{1}{2^*}$ if $N\geq 3$, and $\frac{1}{s} > \frac{\gamma+1}{2\gamma}$ if $N = 2$.
\end{enumerate}

Finally, the pair $(\rho,\uu)$ is a weak solution to
 \begin{align} \label{3.system2}
&\partial_t \rho + \Div(\rho \uu) = 0, \\    
    &\partial_t \big(\rho \uu\big) + \Div \big(\rho \uu \otimes \uu \big) + \nabla \overline{P(\rho)} = \Div \Sbb,
\end{align}
with initial and boundary conditions
\begin{align}
    &\uu(t,x) = 0 &\text{ for $0 < t < T$ and  $\xx \in \partial\Omega$, }\label{4.system2conds0} \\
    &\rho(0,x) = \rho_0(x),\quad (\rho \uu)(0,x) = \m_0(x) &\text{ for }\xx \in \Omega,\label{4.system2conds}
  \end{align}
where $\Sbb = \lambda(\rhoeta)\, (\Div \uu) \mathbb{I} + 2\mu(\rhoeta)\, \mathbb{D}(\uu)$, and satisfies the energy estimate
\begin{multline}\label{5.energy}
  \int_\Omega \Big[ \frac{1}{2} \rho(t) |\uu(t)|^2 + \frac{A}{\gamma-1} \rho(t)^\gamma \Big] \, d\xx \\
  + \int_0^t\int_\Omega \Big[ \mu(\rhoeta) | \nabla \uu|^2 + \Big(\lambda(\rhoeta) + \mu(\rhoeta)\Big) (\Div \uu)^2  \Big] \, d\xx\, dt'  \\
  \leq \int_\Omega \Big[ \frac{1}{2} \rho_0^{-1} |\m_0|^2 + \frac{A}{\gamma-1} \rho_0^\gamma \Big] \, d\xx = E_0 
\end{multline}
for almost every $0 < t < T$.
\end{proposition}
 
 Once more, all that is left then is to show that $\overline{P(\rho)} = A \rho^\gamma$, which is the same as establishing that $\rho^{(\delta)} \rightarrow \rho$ almost everywhere.
  
  This can be achieved in a similar fashion we proceeded previously. However, due to the lack of higher integrability of the density, we are obliged to consider truncations of the sequence $\rho^{(\delta)}$. Following  \cite{FNP}, we choose a function $T \in C^\infty(\R)$ such that
  \begin{enumerate}
    \item $T(z) = z$ for $z \leq 1$,
    \item $T(z) = 2$ for $z \geq 3$, and
    \item $T$ is concave.
  \end{enumerate}
  Furthermore, for any real number $M > 0$, put $T_M : \R \rightarrow \R$ as
  $$T_M(z) = M T\Big( \frac{z}{M} \Big).$$
  
  Notice that $T_M(z) = z$ for $z \leq M$, $T_M(z) = 2M$ for $z \geq 3M$, each $T_M$ is concave and $T_M'(z)$ is uniformly bounded in $0<M<\infty$ and $-\infty < z < \infty$. 
  
  Moreover, as each $\rho^{(\delta)}$ is a renormalized solution to the continuity equation, it holds that
  $$T_M (\rho^{(\delta)})_t + \Div ( T_M(\rho^{(\delta)}) \uu^{(\delta)} ) + (T_M'(\rho^{(\delta)}) \rho^{(\delta)} - T_M(\rho^{(\delta)})) \Div \uu^{(\delta)} = 0,$$
  for any $\delta> 0$ and $M > 0$. Using Proposition~\ref{A.B.prop2} we find that
  $$T_M(\rho^{(\delta)}) \rightarrow \overline{T_M(\rho)} \text{ in } C([0,T];L_\text{weak}^p(\Omega)) \text{ for any } 1 \leq p < \infty.$$
Thus, since $L^p(\Omega)$ is compactly imbedded in $H^{-1}(\Omega)$  for large enough $p$ and $\uu^{(\d)}\to \uu$ weakly in $L^2(0,T;H^{-1}(\Omega))$, we have that
\[
\Div ( T_M(\rho^{(\delta)}) \uu^{(\delta)} ) \to \Div ( \overline{T_M(\rho)}\, \uu )
\]
in the sense of distributions and consequently, for any $M > 0$, the following equation holds also in the sense of distributions
 \begin{equation}\label{5.eq.TMbar}
 \overline{T_M (\rho)}_t + \Div ( \overline{T_M(\rho)}\, \uu ) + \overline{(T_M'(\rho) \rho - T_M(\rho)) \Div \uu} = 0,
 \end{equation}
where $\overline{(T_M'(\rho) \rho - T_M(\rho)) \Div \uu}$ is a weak limit in $L^2((0,T)\X\Omega)$ of $(T_M'(\rho^{(\delta)}) \rho^{(\delta)} - T_M(\rho^{(\delta)}) \Div \uu^{(\delta)}$ as $\d\to 0$.
  
At this point, we realize that by Lemma~\ref{3.lemmarho} we have that $\d\rho^{\beta}\to 0$ strongly in $L^{(\beta+\omega)/\beta}((0,T)\times\Omega)$. Thus, $\overline{P(\rho)}=A\overline{\rho^\gamma}$, from which a direct application of Theorem~\ref{2.thm.EVF} yields the following result.
  
  \begin{lemma} \label{3.effviscflux}
    For any $\varphi \in C_c^\infty((0,T)\times\Omega)$ and any $M > 0$,
        \begin{multline*}
      \int_0^T \int_\Omega \varphi T_M(\rho^{(\delta)}) \Big[ \frac{A(\rho^{(\delta)})^\gamma}{2\mu(\rhodeta)+ \lambda(\rhodeta)} - \Div \uu^{(\delta)} \Big] \, d\xx dt \nonumber \\
               \rightarrow \int_0^T \int_\Omega \varphi \overline{T_M(\rho)} \Big[ \frac{\overline{P(\rho)}}{2\mu(\rhoeta) + \lambda(\rhoeta)} - \Div \uu \Big] \, d\xx dt.
    \end{multline*}
  \end{lemma}

\begin{remark}
Since $2\mu(\rhodeta)+\lambda(\rhodeta)\to 2\mu(\rhodeta)+\lambda(\rhodeta)$ strongly, then it follows directly from Lemma~\ref{3.effviscflux} that
\begin{multline}\label{4.effviscflux2}
      \int_0^T \int_\Omega \varphi T_M(\rho^{(\delta)}) \Big[ A(\rho^{(\delta)})^\gamma - \Big( 2\mu(\rhodeta)+ \lambda(\rhodeta) \Big)\Div \uu^{(\delta)} \Big] \, d\xx dt \\
               \rightarrow \int_0^T \int_\Omega \varphi \overline{T_M(\rho)}\, \Big[ \overline{P(\rho)} - \Big( 2\mu(\rhoeta) + \lambda(\rhoeta) \Big)\Div \uu \Big] \, d\xx dt.
    \end{multline}
\end{remark}
   
With this observation at hand, we deduce the following estimate, which enables for the proof that the limit functions $(\rho,\uu)$ are renormalized solutions of the continuity equation.
 
 \begin{lemma}[Bounds on the oscillation defect measure]\label{5.l.odm}
 	There exists a constant $C$ such that
 	$$\limsup_{\delta \rightarrow 0} \Vert \, T_M(\rho^{(\delta)}) - T_M(\rho) \,\Vert_{L^{\gamma+1}((0,T)\X\Omega)} \leq C,$$
 	for any $M > 1$.    
 \end{lemma}
 \begin{proof}

  Using the inequality $|y - z|^{\gamma+1} \leq (y^\gamma - z^\gamma)(y - z)$, which holds for any non-negative $y$ and $z$, we deduce that
 	\begin{equation*}
 	| T_M(y) - T_M(z)|^{\gamma+1}  \leq (T_M(y)^\gamma - T_M(z)^\gamma)(T_M(y) - T_M(z)),  \quad\quad \,0 \leq y, 0\leq z.
 	\end{equation*}
 From the properties of $T_M$ we see that the function $y\to T_m(y)^\gamma-y^\gamma$ must be non-increasing. Thus, we conclude that
 \begin{equation*}
 	| T_M(y) - T_M(z)|^{\gamma+1}  \leq( y^\gamma - z^\gamma)(T_M(y) - T_M(z)),  \quad\quad \,0 \leq y, 0\leq z.
 	\end{equation*}
 	
Next, since the $y\to y^\gamma$ is convex and $T_m$ is concave, by Lemma~\ref{teo2.11} we see that
\begin{align}
&\limsup_{\d\to 0}\int_0^T \int_\Omega | T_M(\rho^{(\d)}) - T_M(\rho)|^{\gamma+1}\, d\xx dt\nonumber \\
&\qquad\leq \limsup_{\d\to 0}\int_0^T\int_\Omega ((\rho^{(\d)})^\gamma-\rho^\gamma)(T_M(\rho^{(\d)})-T_M(\rho))d\x dt\nonumber\\
&\qquad\leq \limsup_{\d\to 0}\int_0^T\int_\Omega ((\rho^{(\d)})^\gamma-\rho^\gamma)(T_M(\rho^{(\d)})-T_M(\rho))d\x dt \nonumber\\
&\quad\quad\quad\qquad\qquad\qquad+ \int_0^T\int_\Omega (\overline{\rho^\gamma}-\rho^\gamma)(T_M(\rho)-\overline{T_M(\rho)})d\x dt\nonumber\\
&\qquad =  \limsup_{\d\to 0} \int_0^T\int_\Omega   \big[(\rho^{(\delta)})^\gamma T_M(\rho^{(\delta)}) - \overline{\rho^{\gamma}}\,\, \overline{T_M(\rho)} \big] \, d\xx dt\label{5.ineq.convex}
\end{align}

Then, denoting $G(\rhoeta)=2\mu(\rhoeta)+\lambda(\rhoeta)$, we may use \eqref{4.effviscflux2} and the energy inequality \eqref{5.energy} to conclude that
 	\begin{align*}
 	\limsup_{\d\to 0}& \int_0^T \int_\Omega | T_M(\rho^{(\d)}) - T_M(\rho)|^{\gamma+1}\, d\xx dt  \\
 	&\leq \limsup \int_0^T\int_\Omega \Big[ T_M(\rho^{(\delta)}) \Div \uu^{(\delta)}  - \overline{T_M(\rho)}\Div \uu \Big]G(\rhoeta)\, d\xx dt \\
 	&\leq  \|G(\rho)\|_{L^{\frac{\gamma+1}{\gamma-1}}((0,T)\times\Omega)}^{2}  \Big(\sup_{0<\kappa<1} \int_0^T \int_\Omega G(\rhoeta)(\Div \uu^{(\kappa)})^2\, d\xx dt \Big)^{1/2}\\ &\quad\quad\quad\quad\quad\quad\quad \Big(\int_0^T \int_\Omega | \overline{T_M(\rho)} - T_M(\rho)|^{\gamma+1}\, d\xx dt \Big)^{\frac{1}{\gamma+1}} \\
 	&\leq C   \Big(\int_0^T \int_\Omega | \overline{T_M(\rho)} - T_M(\rho)|^{\gamma+1}\, d\xx dt \Big)^{\frac{1}{\gamma+1}},
 	\end{align*}
 	hence the desired conclusion.
 \end{proof}
 
 	 \begin{lemma} [$\rho$ is a renormalized solution]
 		The limit functions $\rho$ and $\uu$ satisfy the continuity equation in the sense of renormalized solutions, i.e., extending $\rho$ and $\uu$ by zero outside $\Omega$, we have
 		\begin{multline}\label{5.eq.renorm}
 		B(\rho)_t + \Div (B(\rho) \Div \uu) + (B'(\rho)\rho - B(\rho)) \Div \uu = 0, \\
 		\text{ in the sense of distributions in $(0,T)\X\R^N$}
 		\end{multline}
 		where $B \in C([0,\infty))\cap C^1((0,\infty))$ satisfies
 		$$|\zeta B'(\zeta)| \leq C(\zeta^{\theta} + \zeta^{\gamma/2}) \quad\quad\quad\quad[\,\forall \zeta > 0\,]$$
 		for some $C>0$ and $0<\theta < \gamma/2$.
 	\end{lemma}
 		The proof of this Lemma consists in regularizing equation \eqref{5.eq.TMbar} through a mollifying operator $S_k$, multiplying the resulting equation by $B'(S_k\overline{T_M(\rho)})$, and letting $k\to \infty$ first and then $M\to \infty$ to find equation \eqref{5.eq.renorm} in the limit, wherein the convergence of this last limit is allowed by the bounds on the oscillation defect measure from Lemma~\ref{5.l.odm}. Since this proof can be carried out line by line as that of lemma 4.4. in \cite{FNP} we omit the details.

 Finally, we deduce the strong convergence of the densities $\rho^{(\delta)}$.
 
 \begin{theorem}
 	$\rho^{(\delta)} \rightarrow \rho$ strongly in $L^1((0,T)\X\Omega)$.
 \end{theorem}
 \begin{proof}
 	Notice that, from the renormalized solution property, we have that
 	$$\rho^{(\delta)} \log \rho^{(\delta)} \rightarrow \overline{\rho \log \rho} \text{ in $C([0,T];L_{weak}^p(\Omega))$ for any $1\leq p<\gamma$. }$$
 	Our goal is again to show that $\rho \log \rho = \overline{\rho \log \rho}$, which will again imply the desired conclusion according to Proposition \ref{teo2.11}.
 	
 	For $M>1$, define the ``entropies''
 	$$B_M(\zeta) = \xi\int_1^\zeta \frac{T_k(\xi)}{\xi^2}\,d\xi,$$
 	so that $b_M(\zeta)$ is an approximation of $B(\zeta) = \zeta \log \zeta$ and $\zeta (B_M)'(\zeta) - B_M(\zeta) = T_M(\zeta)$, which is an approximation of $\zeta \mapsto \zeta$. Applying $B_M$ to both $\rho^{(\delta)}$ and $\rho$ yields
 	\begin{align*}
 	\int_\Omega \big[ B_M&(\rho^{(\delta)}(t,\xx)) - B_M(\rho(t,\xx)) \big] \phi \, d\xx =  \int_\Omega \big[ B_M(\rho_0^{(\delta)}(\xx)) - B_M(\rho_0(\xx)) \big] \phi \, d\xx \\
 	&+\int_0^t \int_\Omega \big[ B_M(\rho^{(\delta)}(t',\xx)) \uu^{(\delta)} - B_M(\rho(t',\xx)) \uu \big] \cdot \nabla \phi(\xx)\, d\xx dt'\\
 	&+\int_0^t \int_\Omega \big[ T_M(\rho(t',\xx)) \Div \uu(t',\xx) - T_M(\rho^{(\delta)}(t',\xx)) \Div \uu^{(\delta)}(t',\xx) \big]\, d\xx dt',
 	\end{align*}
 	where $\phi \in C^\infty(\overline{\Omega})$ and $0<t<T$. By the same convexity argument that leads to \eqref{5.ineq.convex} we have that
 \[
 \overline{\rho^\gamma\, T_M(\rho)} \ge  \overline{\rho^{\gamma}}\, \overline{T_M(\rho)}.
 \]	
Then, taking $\phi \equiv 1$ and letting $\delta$ approach $0$, by Lemma~\ref{3.effviscflux} we have that
 	\begin{align}
 	\int_\Omega \big[ \overline{B_M(\rho)}&(t,\xx) - B_M(\rho(t,\xx)) \big]  \, d\xx \nonumber\\ =& \lim_{\delta \rightarrow 0}\int_0^t \int_\Omega \big[ T_M(\rho) \Div \uu - T_M(\rho^{(\delta)}) \Div \uu^{(\delta)} \big]\, d\xx dt' \nonumber \\
 	=& \int_0^t \int_\Omega \big[ T_M(\rho) - \overline{ T_M(\rho) } \big] \Div \uu \, d\xx dt' \nonumber \\
 	&\qquad + \int_0^t\int_\Omega \Big[ \frac{\overline{P(\rho)}\,\, \overline{T_M(\rho)} - \overline{P(\rho)T_M(\rho)} }{2\mu(\rhoeta) + \mu(\rhoeta)}\, d\xx dt' \nonumber \\
 	&\leq \int_0^t \int_\Omega \big[ T_M(\rho) - \overline{ T_M(\rho) } \big] \Div \uu \, d\xx dt'. \label{3.limitefinal}
 	\end{align}
 	To finalize the proof, we require the following two limits.
 	
 	\textit{Claim \#1:}
 	\begin{equation}
 	 	 \overline{T_M(\rho)} \rightarrow \rho \text{ in $L^p((0,T)\X\Omega)$ for any $1 \leq p < \gamma$ as $M \rightarrow \infty$.} \label{3.claim1}
 	\end{equation}

 	Indeed, notice that
 	$$\Vert \, \overline{T_M(\rho)} - \rho \,\Vert_{L^p((0,T)\X\Omega)}^p \leq \liminf \Vert \, T_M(\rho^{(\delta)}) - \rho^{(\delta)} \,\Vert_{L^p(\Omega)}^p.$$
 	On the other hand $$|T_M(\zeta) - \zeta| \leq \begin{cases} 0 &\text{ if $0 \leq \zeta < M$, and  }\\ (z-M)_+ + M  &\text{ if $M \leq \zeta$,}\end{cases}$$ so that
 	\begin{align*}
 	\Vert \, \overline{T_M(\rho)} - \rho \,\Vert_{L^p((0,T)\X\Omega)}^p &\leq \liminf \frac{1}{M^{\gamma-p}}\int_0^T\int_\Omega 1_{\rho^{(\delta)} \geq M} M^{\gamma - p} |\rho^{(\delta)}|^p\, d\xx dt \\
 	&\leq \frac{1}{M^{\gamma - p}}\liminf \int_0^T \int_\Omega |\rho^{(\delta)}|^\gamma \, d\xx dt\\
 	&\leq \frac{E T}{M^{\gamma - p}} \\ &\rightarrow 0 \text{ as $M \rightarrow \infty$},
 	\end{align*}
 	just as asserted. 
 	
 	\textit{Claim \#2:} 
 	\begin{equation}
 	\overline{B_M(\rho)} \rightarrow \overline{ \rho \log \rho } \text{ in $L^p((0,T)\X\Omega)$ for any $1 \leq p < \gamma$ as $M \rightarrow \infty$.} \label{3.claim2}
 	\end{equation}
 	
 	The argument is almost interchangeable from the one leading to \eqref{3.claim1}. Noting that $|B_M(\zeta) - \zeta \log \zeta| \leq 1_{(M,\infty)}(\zeta) \zeta \log \zeta$ for $\zeta > 0$ and $M > 1$, we see that for $p < \sigma < \gamma$,
 	\begin{align*}
 	\Vert \, \overline{B_M(\rho)} -&\overline{ \rho \log \rho } \,\Vert_{L^p((0,T)\X\Omega)}^p \leq \liminf  \Vert \, B_M(\rho^{(\delta)}) - \rho^{(\delta)} \log \rho^{(\delta)} \,\Vert_{L^p((0,T)\X\Omega)}^p\\
 	&\leq \frac{1}{ (M \log M)^{\sigma - p}}\int_0^T\int_\Omega 1_{\rho^{(\delta)} \geq M} M^{\sigma - p} |\rho^{(\delta)} \log \rho^{(\delta)}|^p\, d\xx dt \\
 	&\leq \frac{C}{(M \log M)^{\sigma - p}}\liminf \int_0^T \int_\Omega |\rho^{(\delta)}|^\gamma \, d\xx dt\\
 	&\leq \frac{C E T}{(M \log M)^{\sigma - p}} \rightarrow 0 \text{ as $M \rightarrow \infty$},
 	\end{align*}
which shows the claim.
 	
 	In virtue of both \eqref{3.claim1} and \eqref{3.claim2}, we can pass $M \rightarrow \infty$ in \eqref{3.limitefinal} to conclude that
 	$$\int_\Omega \big[ \overline{\rho \log \rho}(t,\xx) - \rho(t,\xx) \log \rho(t,\xx)\big] \, d\xx \leq 0$$
 	for any $0 \leq t \leq T$. This implies that $\overline{\rho \log \rho} = \rho \log \rho$, which, as explained in the beginning of this proof, furnishes the desired result.
 \end{proof}

The proof of Theorem~\ref{T:1.1} is thereby complete.

\begin{appendix}

 \section{On the $C([0,T]; E_\text{weak})$ spaces.}\label{A.ApdxA}
   
   Let $E$ be a Banach space and $E^*$ its dual. We denote by $E_\text{weak}^*$ the space $E^*$ endowed with the weak-$\star$ topology $\sigma(E,E^*)$. Similarly, we denote by $E_{weak}$ the space $E$ endowed with the weak topology $\sigma(E^*,E)$.
   
By $C([0,T]; E_\text{weak})$, we understand the set of functions $u : [0,T] \rightarrow E_\text{weak}$ which are continuous. Provided that $E$ is separable, then the weak-$\star$ topology is metrizable on bounded sets in $E^*$ and the space $C([0,T]; E_\text{weak}^*)$ is also metrizable on bounded sets. Thus, combining the Banach-Alaouglu theorem with the Arzel\`{a}-Ascoli theorem yields the following (see corollary 2.1 in \cite{F}).
   
   \begin{proposition} \label{A.B.prop2}
   Let $E$ be a separable Banach space. Assume that $v_n::[0,T]\to E^*$, $n=1,2,...$ is a sequence of measurable functions such that
   \[
   \sup_{t\in[0,T]}\| v_n(t)\|_{E^*}\leq M\qquad\text{uniformly in }n=1,2,...
   \]
Moreover, let the family of (real) functions
\[
\langle v_n, \Phi \rangle:t\to \langle v_n(t),\Phi\rangle,\qquad t\in [0,T], \quad n=1,2,...
\]
be equicontinuous for any fixed $\Phi$ belonging to a dense subset in the space $E$.

Then, $v_n\in C([0,T];E_{weak}^*)$ for any $n=1,2,...$, and there exists $v\in C([0,T];E_{weak}^*)$ such that 
\[
v_n\to v in C([0,T];E_{weak}^*)\text{ as }n\to\infty,
\]
passing to a subsequence as the case may be.
   \end{proposition}

This result is particularly useful if the space $E$ is reflexive.

   \begin{remark}\label{A.A.remark1}
   In the text, we also encountered the space $C([0,T];C^\infty(\overline{\Omega}))$, which does not fall precisely in the hypotheses of this appendix. However, $C^\infty(\overline\Omega)$ is a Fréchet space, in particular, a metric space, so its topology is straightforward to define and it is actually metrizable. (Recall that $u_n \rightarrow u$ in $C^\infty(\overline{\Omega})$ if, and only if, $u_n \rightarrow u$ in $C^k(\overline{\Omega})$ for every $k \geq 0$).
      \end{remark}

\section{The $\Delta^{-1}\nabla $ operator and some commutator estimates involving Riesz transforms}\label{ApdxC}
   
   Let $N \geq 1$ be an integer. By the symbol $\Delta^{-1}\nabla$, we will understand the operator which maps $L^p(\mathbb{R}^N)$ into the homogeneous Sobolev space $\dot{W}^{1,p}(\mathbb{R}^N;\mathbb{R}^N)$ for any $1<p<\infty$ and is given by the formula
   $$(\Delta^{-1} \nabla f)(\xx) = \Big[ \Delta^{-1} \Big(\frac{\partial f}{\partial x_1}\Big)(\xx), \ldots, \Delta^{-1} \Big(\frac{\partial f}{\partial x_N}\Big)(\xx) \Big],$$
   for any $f \in S(\mathbb{R}^N)$; equivalently, and probably more clearly, $\Delta^{-1} \nabla$ is operator whose the Fourier multiplier is
   $$(\Delta^{-1} \nabla f)_{j}\widehat{\,}\,(\xi) = -\frac{i\xi_j}{|\xi|^2} \widehat f(\xi),$$
   where $1 \leq j \leq N$ and again $f \in S(\R^N)$.
   
  In virtue of the Sobolev inequality, 
  $$\Delta^{-1} \nabla : L^p(\mathbb{R}^N) \rightarrow L^{p*}(\mathbb{R}^N;\mathbb{R}^N) \text{ constinuously},$$ provided that $1 < p < N$, where $\frac{1}{p^*} = \frac{1}{p} - \frac{1}{N}$. Consequently, Morrey's theorem asserts that, provided that $1<q<N<p<\infty$,
  $$\Delta^{-1} \nabla: (L^q \cap L^p) (\R^N) \rightarrow (C^\alpha \cap L^\infty)(\R^N;\R^N)$$
  for $\alpha = 1 - N/p$. Observe that, in our case, we will only apply $\Delta^{-1} \nabla$ to functions supported in $\Omega$, so that many of these conclusions may be strengthened as follows.
   
   \begin{proposition} \label{A.C.1}
   Let $\Omega \subset \R^N$ be a bounded open set and, for any $1 < p < \infty$, consider $L^p(\Omega)$ as subspace of $L^p(\R^N)$ by extending its elements to be zero outside $\Omega$. Then,
  \begin{enumerate}
    \item there exists a constant $C = C(p, \Omega)$ such that
    $$\Vert \Delta^{-1} \nabla u \Vert_{W^{1,p}(\Omega)} \leq C \Vert u \Vert_{L^p(\Omega)} \text{ for any $u \in L^p(\Omega)$};$$
    \item for any $1<p<N$ and $1\leq q<\infty$ satisfying $\frac{1}{q} \geq \frac{1}{p} - \frac{1}{N}$, there exists a constant $C = C(p, q, \Omega)$ such that
    $$\Vert \Delta^{-1} \nabla u \Vert_{W^{1,q}(\Omega)} \leq C \Vert u \Vert_{L^p(\Omega)} \text{ for any $u \in L^p(\Omega)$};$$
    \item if $p > N$, there exists a constant $C = C(p, \Omega)$ such that
    $$\Vert \Delta^{-1}\nabla  u \Vert_{C^\alpha(\overline\Omega)} \leq C \Vert u \Vert_{L^p(\Omega)} \text{ for any $u \in L^p(\Omega)$},$$
    where $\alpha = 1 - N/p$.
  \end{enumerate}
   \end{proposition}
   
   An important feature of this operator $\Delta^{-1} \nabla$ is that $\nabla \Delta^{-1} \nabla$ can be seen as the matrix
   $$\nabla \Delta^{-1} \nabla = \big[ R_j R_k \big]_{1\leq j, k \leq N},$$
   where $R_j$ denotes the $j^{\text{th}}$ Riesz transform; that is, the Fourier operator with multiplier
   $$(R_j  f)\widehat{\,\,}\,(\xi) = -\frac{i\xi_j}{|\xi|} \widehat f(\xi) \quad\quad [\, f \in S(\R^N)\, ].$$
   This allows us to apply the  results by Coifman, Rochberg and Weiss \cite{CRW} and by Coifmann and Meyer  \cite{CM} on the regularity of the commutators involving Riesz transforms. The following theorem plays an important role in our work.
   
    \begin{theorem} \label{A.C.2}
      If $1\leq j, k \leq N$, and $b$ and $f \in S(\mathbb{R}^N)$, let us define the commutator
      $$\big[ b, R_j R_k \big] f(x) = b(x) (R_jR_k f)(x) - (R_j R_k (bf))(x).$$
      Then, 
      \begin{enumerate}
        \item (Coifman--Rochberg--Weiss) for $1< p < \infty$, there exists a constant $C = C(p)$ such that
        $$\big\Vert \big[b, R_jR_k \big] f\big\Vert_{L^p(\R^N)} \leq C \Vert b \Vert_{BMO(\R^N)} \Vert f \Vert_{L^p(\R^N)}; \text{ and}$$
        \item (Coifman--Meyer) if $1 < p, q, r <\infty$ with $\frac{1}{r} = \frac{1}{p} + \frac{1}{q}$, then there exists a constant $C = C(p,q,r)$ such that
        $$\Vert \nabla[b, R_jR_k] f\big\Vert_{L^r(\mathbb{R}^N)} \leq C \Vert \nabla b \Vert_{L^p(\R^N)} \Vert f \Vert_{L^q(\R^N)}.$$
      \end{enumerate}
    \end{theorem}
    
    Another crucial and related result is the result which is a particular case of the the div--curl lemma (see \cite{F}, corollary 6.1).

     \begin{theorem} \label{A.C.3}
      Assume that $\Omega \subset \R^N$ is an open set and $1<p,q, r < \infty$ satisfy
      $$ \frac{1}{p} + \frac{1}{q} = \frac{1}{r}.$$
      Moreover, let $(f_n)$ be a sequence in $L^p(\Omega)$ and $(v_n)$ be a sequence in $L^q(\Omega;\R^N)$ for which
      \begin{align*}
        f_n \rightharpoonup f &\text{ weakly in $L^p(\Omega)$, and} \\
        v_n \rightharpoonup v &\text{ weakly in $L^q(\Omega;\R^N)$}.
      \end{align*}
      Then
      \begin{align*}
        (\nabla \Delta^{-1} \nabla)(f_n) &v_n - (\nabla \Delta^{-1} \Div)(v_n) f_n \\
        &\rightharpoonup   (\nabla \Delta^{-1} \nabla)(f) v - (\nabla \Delta^{-1} \Div)(v) f \text{ weakly in $L^r(\Omega;\R^N)$}.
      \end{align*}
     \end{theorem}
     
\section{Weak convergence and convexity}

Let us state without proof the following general result which proved itself be very useful to show strong convergence of the sequences of densities considered in this text (see \cite[Theorem 2.11]{F}).
\begin{lemma}\label{teo2.11}
Let $O\subseteq \mathbb{R}^N$ be a measurable set and $\{ \mathbf{v}_n\}_{n=1}^\infty$ a sequence of functions in $L^1(O;\mathbb{R}^M)$ such that
\[
\mathbf{v}_n\to\mathbf{v} \text{ weakly in }  L^1(O;\mathbb{R}^M).
\]
Let $\Phi:\mathbb{R}^M\to (-\infty,\infty]$ be a lower semi-continuous convex function such that $\Phi(\mathbf{v}_n)\in L^1(O)$ for any $n$ and
\[
\Phi(\mathbf{v}_n)\to\overline{\Phi(\mathbf{v})} \text{ weakly in } L^1(O).
\]

Then,
\[
\Phi(\mathbf{v})\leq \overline{\Phi(\mathbf{v})} \text{ a.a. on } O.
\]

If, moreover, $\Phi$ is strictly convex on an open convex set $U\subseteq \mathbb{R}^M$, and
\[
\Phi(\mathbf{v})=\overline{\Phi(\mathbf{v})} \text{ a.a. on }O,, 
\]
then,
\[
\mathbf{v}_n(\mathbf{y})\to\mathbf{v}(\mathbf{y}) \text{ for a.e. } \mathbf{y}\in\{ \mathbf{y}\in O:\mathbf{v}(\mathbf{y})\in U\},
\]
extracting a subsequence as the case may be.
\end{lemma}
\end{appendix}

\end{document}